\RequirePackage{fix-cm}
\documentclass[reqno,11pt]{amsart}

\usepackage{mathptmx}      
\usepackage[active]{srcltx} 
\usepackage{amsmath}
\usepackage{amsfonts}
\usepackage{color,ulem,hyperref}
\newcommand{\clrr}{\relax}
\usepackage[paper=a4paper,margin=0.8in, ]{geometry}
\usepackage{amsthm}
\newtheorem{theorem}{Theorem}[section]
\newtheorem{lemma}[theorem]{Lemma}
\newtheorem{definition}[theorem]{Definition}
\newtheorem{remark}[theorem]{Remark}
\newtheorem{example}[theorem]{Example}

\newtheorem{assumption}[theorem]{Assumption}
\numberwithin{equation}{section} \numberwithin{theorem}{section}
\DeclareMathOperator{\Lip}{Lip}
\newcommand{\dohat}{^{\wedge}}
\newcommand{\unhat}{^{\vee}}
\newcommand{\PP}{\mathbb P}
\newcommand{\QQ}{\mathbb Q}
\newcommand{\TT}{\mathbb T}
\newcommand{\RR}{\mathbb R}
\newcommand{\del}{\delta}
\newcommand{\eps}{\varepsilon}
\newcommand{\fr}[1]{\frac1{#1}}
\newcommand{\tfr}[1]{\tfrac1{#1}}
\newcommand{\half}{{\mathchoice{\tfrac{1}{2}}{\frac{1}{2}}%
{\frac{1}{2}}{\frac{1}{2}}}}
\newcommand{\ddt}{\frac{d \hfill}{dt}}
\newcommand{\tddt}{\tfrac{d \hfill}{dt}}
\newcommand{\eqdef}{\mathrel{:=}}
\newcommand{\prt}{\partial}
\newcommand{\grad}{\nabla}
\newcommand{\gradu}{\grad_{\!u}}
\newcommand{\dd}[1]{\frac{d\hfill}{d#1}}
\newcommand{\eval}[1]{{\big|_{#1}}}
\newcommand{\ip}[2]{\left<#1,#2\right>}
\newcommand{\alp}{\alpha}
\newcommand{\lb}{\left[}
\newcommand{\rb}{\right]}
\newcommand{\la}{\left|}
\newcommand{\ra}{\right|}
\newcommand{\lB}{\left\{}
\newcommand{\rB}{\right\}}
\newcommand{\lp}{\left(}
\newcommand{\rp}{\right)}
\newcommand{\ppp}[2]{\frac{\prt #1\hfill}{\prt #2\hfill}}
\newcommand{\tppp}[2]{\tfrac{\prt #1\hfill}{{\prt #2}\hfill}}
\newcommand{\fakeq}{\phantom{=\;}} 
\newcommand{\vertiii}[1]{{\left\vert\mkern-1.5mu\left\vert\mkern-1.5mu\left\vert #1 
    \right\vert\mkern-1.5mu\right\vert\mkern-1.5mu\right\vert}}
\newcommand{\vertiv}[1]{{\left\vert\mkern-1.5mu\left\vert\mkern-1.5mu\left\vert\mkern-1.5mu\left\vert #1 
    \right\vert\mkern-1.5mu\right\vert\mkern-1.5mu\right\vert\mkern-1.5mu\right\vert}}
\newcommand{\pseudo}{\psi} 
\newcommand{\alfven}{Alfv\'en\ }

\newcommand{\deA}{\del_\textnormal{\tiny A}}
\newcommand{\epM}{\eps_\textnormal{\tiny M}}
\newcommand{\ccdot}{{\hspace{-0.15mm}\cdot }}
\newcommand{\nc}{\nabla \ccdot }
\newcommand{\vu}{{\mathbf u}}
\newcommand{\cn}{ \!\cdot\!\! \nabla}
\newcommand{\sqkel}{\sqrt{k^2+4\eps^2\ell^2}}
\def\at{\newline}
\def\and{\;\&\;}
\newcommand{\eemail}[1]{\newline\textit{E-mail}: #1}
\begin{document}
\title[{Three-scale Singular Limit of Evolutionary PDEs}]{Three-Scale Singular Limits of Evolutionary PDEs}
\author{Bin Cheng \and   Qiangchang Ju  \and Steve Schochet}
%
%
\date{14-04-2019\newline
B. Cheng \at
Department of Mathematics, University of Surrey, Guildford, GU2 7XH, United Kingdom
              \eemail{b.cheng@surrey.ac.uk}  
           \newline           
           Q. Ju 
           \at
         Institute of Applied Physics and Computational Mathematics, P.O. Box 8009, Beijing 100088, China
           \eemail{ju\_qiangchang@iapcm.ac.cn}
          \newline 
        S. Schochet \at School of Mathematical Sciences, Tel-Aviv University, Tel Aviv 69978, Israel 
     \eemail{schochet@tauex.tau.ac.il}
}
\maketitle

\begin{abstract}
  Singular limits of a class of evolutionary systems of partial differential equations having two
  small parameters and hence three time scales are considered. Under appropriate conditions
  solutions are shown to exist and remain uniformly bounded for a fixed time as the two parameters
  tend to zero at different rates. A simple example shows the necessity of those conditions in order
  for uniform bounds to hold. Under further conditions the solutions of the original system tend to
  solutions of a limit equation as the parameters tend to zero.

\keywords{Keywords: singular limit  \and multiple time scales. 35B25, 35L45. }
\end{abstract}

\section{Introduction}
Many physical systems contain several small parameters, such as the Mach number, \alfven number,
Froude number, Rossby number, etc. When these parameters are considered to have fixed ratios to one
another then the system has two time scales, one induced by the terms containing the small
parameters and the other coming from the order-one terms in the equation. The classical theory of
singular limits for evolutionary partial differential equations (PDEs)
(\cite{klainerman81:singlim,mr665380,Majdabook84,schochet88:asympt,scho94f,mr1459589,MR1661025} and
numerous papers on particular systems, e.g. \cite{metivier01:euler}) was developed to treat this
case.  In order to determine the behavior of solutions when two physical parameters tend to zero in
a different manner it is necessary to develop an analogous theory for systems with three time
scales. The systems to be considered here have the form
\begin{equation}
  \label{eq:3pde}
  A_0(\eps u)u_t+\sum_{j=1}^d A_j(u)u_{x_j}+\tfr{\del}\mathcal L u+\tfr{\eps}\mathcal Mu=F(t,x,u),
\end{equation}
where $\eps$ and $\del$ are small parameters.
As in the theory of two-scale singular limits, the system without the large terms is assumed to be
symmetric hyperbolic, and $\mathcal L$ and $\mathcal M$ are assumed to be antisymmetric
constant-coefficient differential or pseudodifferential operators of order at most one. As for two-scale
singular limits \cite{klainerman81:singlim,MR1661025}, parabolic terms of size $O(1)$ could be added to the
right side of~\eqref{eq:3pde}, although the complications such terms induce would be greater in the
three-scale case.

The fundamental discovery of Klainerman and Majda \cite{klainerman81:singlim,klainerman82:compress}
for two-scale singular limits was that the presence of the small parameter in the matrix~$A_0$,
which occurs naturally in the normalized equations for low Mach number fluid flow, induces a
delicate balance. As they showed, this ensures that solutions of \eqref{eq:3pde} with $\del=1$,
having fixed initial data belonging to a Sobolev space of sufficiently high index, exist for a time
independent of the small parameter~$\eps$ and satisfy bounds independent of that parameter, without
the need for additional conditions on the large terms or the initial data, such as those assumed in
\cite{mr665380,schochet88:asympt} to treat the case when $A_0$ depends on $u$ rather than $\eps
u$. Whenever the small parameter~$\del$ in \eqref{eq:3pde} is not asymptotically smaller
than $\eps$, i.e., when $\del\ge c\eps$ for some arbitrarily small positive constant~$c$, then the
Klainerman-Majda balance is essentially preserved and their uniform existence result remains valid
and requires only cosmetic changes to the proof. Similarly, when $A_0$ is a constant matrix, as in
the rotating shallow {\clrr  water} equations (\cite[Equation (2.2)]{majda03:_system_multis_model_tropic}) then the
Klainerman-Majda uniform existence result remains valid for arbitrary $\del$ and $\eps$.

Hence we will be concerned here with the case when $A_0$ does depend nontrivially on $\eps u$, and 
\begin{equation}\label{eq:delsmall}
0<\delta\ll \eps\ll 1.
\end{equation}
Our first main result is a uniform existence theorem under two additional assumptions: The first condition
is that 
\begin{equation}
  \label{eq:deleps}
  \del\ge c\,\eps^{1+\fr{s_0}}
\end{equation}
for some positive constant~$c$, where 
\begin{equation}
  \label{eq:s0}
  s_0\eqdef \lfloor\tfrac d2\rfloor+1
\end{equation}
is the Sobolev embedding exponent in dimension~$d$. The second condition is that the initial
data~$u_0(x,\eps,\del)$ are uniformly bounded in the Sobolev space~$H^{s_0+1}(\mathbb{D})$ and are
``well-prepared'' in the usual sense that the initial time derivative
\begin{equation}\label{eq:ut0}
u_t(0,x)\eqdef A_0(\eps u_0)^{-1}\lb F({\clrr  0},x,u_0)-\sum_{j=1}^d
A_j(u_0)(u_0)_{x_j}-\tfr\del\mathcal Lu_0-\tfr\eps\mathcal Mu_0\rb
\end{equation}
is uniformly bounded in $H^{s_0}(\mathbb{D})$, with the domain $\mathbb{D}$ being either the whole
space $\mathbb{R}^d$ or the torus $\mathbb{T}^d$. {\clrr Examples of initial data satisfying this condition
are given in \eqref{eq:u0} below.} For convenience, we shall {\clrr henceforth} omit the spatial domain
in integrals and function spaces throughout the paper.
Although \eqref{eq:deleps} limits how small $\del$ can be
compared to~$\eps$, it is consistent with the scaling~\eqref{eq:delsmall} that violates the
Klainerman-Majda balance. Moreover, both conditions are necessary, at least for obtaining uniform
bounds on solutions of general systems without Klainerman-Majda balance, as will be shown via 
a simple explicit example.

Our other main result is a convergence theorem showing, under the additional assumptions described
below, that as $\eps$ and $\del$ both tend to zero solutions of \eqref{eq:3pde} whose initial data
converge in $H^{s_0+1}$ tend to the solution of a certain limiting equation. The framework of the
convergence theorem is the same as for two-scale singular limits: The bounds of the existence result
yield compactness, which implies that every sequence of $\eps$ and $\del$ tending to zero while
obeying \eqref{eq:deleps} has a subsequence for which the solution converges, and convergence
without restricting to such subsequences is obtained by showing that the limit satisfies a limit
equation for which solutions of initial-value problems are unique. However, both the form and the
derivation of the limit equation are more complicated for three-scale singular limits. For the
two-scale singular limit obtained when $\del\equiv1$, the limit equation is obtained by decomposing
\eqref{eq:3pde} into the projections onto the null space of $\mathcal M$ and onto its orthogonal
complement, multiplying the latter by $\eps$ and taking the limits of the results. However, in order
to obtain the limit equation for the three-scale singular limit in which
\eqref{eq:delsmall}\textendash\eqref{eq:deleps} holds it is necessary to use perturbation theory to
compute some number of terms of the 
power series in the small parameter $\mu=\frac{\del}{\eps}$ 
for the eigenvalues and eigenspaces of $\mathcal L+\mu
\mathcal M$ in Fourier space.
The number of
terms required and the resulting limit equation depend on the relationship between $\del$ and $\eps$ as
they both tend to zero. In order to obtain convergence without restricting to subsequences it is
necessary to restrict the relationship between $\del$ and $\eps$ so as to obtain a specific limit
equation. This requires the additional assumption that
for some integer $s\ge
s_0$ either
\begin{equation}
  \label{eq:delepsrel1}
\frac{\del}{\eps^{1+\fr s}}\to C>0 \qquad\text{as $\eps$ and $\del$ tend to zero}
\end{equation}
or
\begin{equation}
  \label{eq:delepsrel2}
\frac{\del}{\eps^{1+\fr {s}}}\to \infty \qquad\text{and}\qquad \frac{\del}{\eps^{1+ \fr{s+1}}}\to 0 
\qquad\text{as $\eps$ and $\del$ tend to zero,}
\end{equation}
either of which implies \eqref{eq:delsmall}\textendash\eqref{eq:deleps} hold.  Note that if
$\frac{\del}{\eps^{1+\fr r}}\to C>0$ for some $r>s_0$ that is not an integer then
\eqref{eq:delepsrel2} holds with $s=\lfloor r\rfloor$. The limit equation is different for different
values of $s$ and even for different values of $C$ in \eqref{eq:delepsrel1}, but is the same for all
$r$ in $(s,s+1)$.  The reason why the limit equation depends on $C$ is that when
\eqref{eq:delepsrel1} holds then the limit equation contains a term~{\clrr $T_{\lim}$} arising from the 
power series expansion in $\delta$
of
$\fr{\del}(\mathcal L+\mu\mathcal M)$. Moreover, although both $\mathcal L$ and $\mathcal M$ are both
bounded operators from $H^1$ to $L^2$ {\clrr it turns out that $T_{\lim}$} 
may not be, {\clrr  as will be explained in Definition~\ref{def:proj} and
  Remark~\ref{rem:tlim} below.} Such
terms do not occur in two-scale singular limits.  As a result, the second time derivative of the
limit solution may not belong to $L^2$, although the limit process ensures that its first time
derivative does belong to $L^2$.

After presenting the example showing the necessity of our conditions for obtaining uniform bounds in
\S2, the uniform existence theorem will be formulated precisely and proven in \S3, and the 
convergence theorem will
be formulated precisely and proven in \S4. Some simple examples of the perturbation procedure and the limit
equations will also be presented in that section. 
In forthcoming work the results here will be applied to the problem that
motivated this research, namely the simultaneous zero \alfven number and zero Mach
number limit of the scaled compressible inviscid MHD equations \def\Sem{S_{\epM}} \def\vb{{\mathbf
    b}} \def\pz{\prt_z} \def\vez{{\mathbf e_z}}
\begin{subequations}\label{MHDsy}
\begin{align}
\label{MHDsy:r}a(1+\epM r)\,\big(&\prt_t r+\vu\cn r)
+R(r,\epM) \nc \vu+\epM^{-1}\nc\vu=0 ,\\
\label{MHDsy:vu}(1+\epM r)\big(&\prt_t\vu+\vu\cn \vu\big)+R(r,\epM)\grad r
+\epM^{-1}\grad r+ \nabla{\frac{|\vb|^2}2}-\vb\cn\vb  =\deA^{-1} (\pz\vb-\nabla b_3), \\
\label{MHDsy:vb}&\prt_t\vb+\vu\cn\vb
+(\nc\vu)\vb-\vb\cn\vu  =\deA^{-1}( \pz\vu- \vez\,\nc\vu),\qquad\nc\vb=0,
\end{align}
\end{subequations}
where 
the small parameters $\epM$ and $\deA$ are respectively the Mach number and \alfven number,
the fluid density is $1+\epM r$, its velocity is $\vu$, the magnetic field is $\vez+\deA \vb$ with
$\vez$ being the unit vector in the $z$-direction, and the coefficient functions $a$ and $R$ depend
on the constitutive relation giving the fluid pressure as a function of its density.

\section{Example}

Consider the system
\begin{equation}
  \label{eq:uvw}
  a(\eps w)u_t-\tfr\del v=0,\quad a(\eps w)v_t+\tfr\del u=0,\quad w_t=0,
\end{equation}
which has the form~\eqref{eq:3pde}, together with the initial data
\begin{equation}
  \label{eq:uvw0}
  u(0,x)=u_0\eqdef \del,\quad v(0,x)=v_0\eqdef 0, \quad w(0,x)=w_0(x)
\end{equation}
that satisfy the condition that the initial time derivative be uniformly bounded.
The system could be turned into one in which the large terms involve derivatives {\clrr  with
  respect to an additional spatial variable~$y$} by replacing the
terms $-\tfr\del v$ and $\tfr\del u$ by $\tfr\del v_y$ and $\tfr\del u_y$, respectively, and
changing $u_0$ to $\del\cos y$. A term containing $\tfr\eps$ could also be added.

It will be convenient to write the solution to \eqref{eq:uvw}\textendash\eqref{eq:uvw0} in terms of 
\begin{equation}
  \label{eq:z}
  z\eqdef u+iv,
\end{equation}
which satisfies
\begin{equation}
  \label{eq:zq}
  z_t={\clrr -}\frac{iz}{\del a(\eps w_0(x))},\quad z(0,x)=\del
\end{equation}
since $w(t,x)=w_0(x)$.
The solution to \eqref{eq:zq} is
\begin{equation}
  \label{eq:zsol}
  z(t,x)=\del e^{{\clrr -}\frac{it}{\del a(\eps w_0(x))}}.
\end{equation}
Differentiating \eqref{eq:zsol} or its derivatives with respect to~$t$ produces a term containing a
factor~$\tfr\del$, while differentiating with respect to~$x$ produces a term containing a
factor~$\frac\eps\del$ since the $x$-dependence in the exponent of~\eqref{eq:zsol} lies inside
$a(\eps \cdot)$. Taking into account the factor of $\del$ in \eqref{eq:zsol} that comes from the
initial condition, this shows that
\begin{equation}
  \label{eq:zest}
  \prt_x^\ell \prt_t^k z= \tfrac{\eps^\ell}{\del^{k+\ell-1}}(z_{k,\ell}(t,x)+
o(1)
)
\end{equation}
for some function~$z_{k,\ell}$ that is not identically zero provided that both $a$ and $w_0$
genuinely depend on their argument.

The standard existence theory for symmetric hyperbolic systems in spatial dimension~$d$ requires obtaining a bound on the
$H^{s_0+1}$ norm of solutions. The system~\eqref{eq:uvw} can be considered to be a system in any
dimension, and estimate~\eqref{eq:zest} implies that the solution
of~\eqref{eq:uvw}\textendash\eqref{eq:uvw0} will be uniformly bounded in $H^{s_0+1}$  only when
$\tfrac{\eps^{s_0+1}}{\del^{s_0}}$ is bounded, which requires that~\eqref{eq:deleps}
{\clrr must} hold.

Moreover, if the condition that the initial dat{\clrr a must} be well prepared is dropped then the initial value
of $u$ {\clrr  in \eqref{eq:uvw0}} can be $1$ rather than $\del$, which makes \eqref{eq:zest} more singular by one power of
$\del$. The condition that the $H^{s_0+1}$ norm {\clrr  of the solution} be uniformly bounded then requires that 
$\tfrac{\eps^{s_0+1}}{\del^{s_0+1}}$ be bounded, i.e., that $\del\ge c\eps$, so no general result
beyond the Klainerman-Majda balance is then possible.

\section{Uniform Existence Result}

\subsection{Scaling}

Estimate~\eqref{eq:zest} implies that the derivatives of solutions~$(u,v,w)$
of~\eqref{eq:uvw} satisfy corresponding estimates, except that $(u,v,w)$ itself and its pure spatial
derivatives are no smaller than $O(1)$ because that is the size of the component~$w$. These estimates
suggest that the appropriate norm of solutions of ~\eqref{eq:3pde} to estimate would be
\begin{equation}\label{eq:uvwnorm}
\|u\|_{H^{s_0+1}}+
\sum_{k=1}^{s_0+1}\sum_{0\le|\alp|\le s_0+1-k}\tfrac{\del^{k+|\alp|-1}}{\eps^{|\alp|}}
\|D^\alp\prt_t^k u\|_{L^2},
\end{equation}
where 
as usual 
$D^\alp$ denotes the spatial derivative
$\prod_{j=1}^d\prt_{x_j}^{\alp_j}$ of order $|\alp|\eqdef \sum_{j=1}^d\alp_j$.
Although our method indeed allows us to estimate the weighted norm \eqref{eq:uvwnorm} of solutions,
doing so requires keeping an exact count of the spatial derivatives appearing in instances of the
Gagliardo-Nirenberg inequalities \eqref{eq:gn} below. In order to avoid the need to count spatial
derivatives we will instead 
{
perform a simplified estimate by using
} 
weights that depend only on the number of time derivatives, with the
weight of the term $\prt_t^k u$ and its spatial derivatives equal to $\eps^{k}$. 
These weights equal their counterparts in~\eqref{eq:uvwnorm} for the highest spatial 
derivative of $\prt_t^ku$, under the assumption that equality holds in~\eqref{eq:deleps}.
For lower-order spatial derivatives when that equality holds, 
or for all cases when strict inequality holds in~\eqref{eq:deleps}, the
weights we use 
are  smaller than their counterparts in~\eqref{eq:uvwnorm}. 
Hence 
the simplified estimate will be somewhat weaker
than the estimate that would be obtained using~\eqref{eq:uvwnorm}.  With one exception this
difference is of little importance, because estimates of norms of time derivatives of solutions
weighted by small constants are simply a means of obtaining an unweighted estimate for the spatial
norms of solutions. The one exception is that the $L^2$ norm of $u_t$ in \eqref{eq:uvwnorm} has
weight one and so yields a uniform bound, while the $L^2$ norm of $u_t$ in the simplified scheme has
weight~$\eps$ and so does not yield a uniform bound. Obtaining a uniform bound for some norm of
$u_t$ is important for the convergence theory, and it will turn out that the time evolution of the
unweighted $L^2$ norm of $u_t$ can be estimated in terms of the norms appearing in the simplified
estimates, so we will simply adjoin the unweighted $L^2$ norm of $u_t$ to the simplified scheme of
estimates.

However, as is 
common
in the theory of hyperbolic systems, 
we must modify the standard $L^2$ and $H^s$ norms
to include the coefficient matrix~$A_0(\eps u)$ of the time-derivative term in the
PDE~\eqref{eq:3pde}, with the argument $\eps u$ of $A_0$ taken from some solution
to~\eqref{eq:3pde}.  We therefore define
\begin{equation}\label{eq:a0norms}
\begin{aligned}
  &\ip{v}{w}_{A_0}\eqdef \int v^T\!A_0(\eps u) w\,dx,\;\; \|v\|_{0,A_0}\eqdef\sqrt{\ip{v}{v}_{A_0}},
\;\;
\|v\|_{\ell,A_0}\eqdef \sqrt{\sum_{0\le|\alp|\le \ell}\|D^\alp v\|_{0,A_0}^2},
\\
&\vertiii{u}_{s,\eps,A_0}\eqdef\sqrt{\sum_{k=0}^s \eps^{2k}\|\prt_t^k u(t,\cdot)\|_{s-k,A_0}^2},
\quad
   \vertiv{u}_{s,\eps,A_0}\eqdef\sqrt{\vertiii{u}_{s,\eps,A_0}^2+\|u_t\|_{0,A_0}^2}\,.
\end{aligned}
\end{equation}
The
corresponding quantities with the subscript~$A_0$ omitted will denote the standard inner product and
norms in which $A_0$ is replaced by the identity matrix. Assumption~\ref{sum:hyp}
together with the estimates to be obtained will ensure that the two are equivalent 
for the time intervals considered here. The definitions in~\eqref{eq:a0norms} are a slight abuse of
notation, 
both since the argument of $A_0$ is usually not given explicitly but must be understood from the
context, and
because the value of $\vertiii{u}_{s,\eps,A_0}$ and $\vertiv{u}_{s,\eps,A_0}$ 
at a given time does not depend solely on the value
of~$u$ at that time on account of the inclusion of time derivatives.

{\clrr 
\begin{remark}\label{rem:exist}
  The standard existence theorem for symmetric hyperbolic systems (\cite[Ch. 2, Theorem 2.1]{Majdabook84})
  shows that there exists a unique solution, for some time that may depend on $\del$ and $\eps$, to
  the initial-value problem consisting of \eqref{eq:3pde} together with an initial condition
  $u(0,x)=u_0(x,\del,\eps)\in H^{s_0+1}$. Moreover (\cite[Ch. 2, Theorem 2.2]{Majdabook84}), that solution continues to exist as long as its
  $H^{s_0+1}$ norm remains finite. Hence in order to prove that the time of existence can be taken
  to be independent of $\del$ and $\eps$ it suffices to obtain a uniform bound on the $H^{s_0+1}$
  norm of the solution. The proof of the existence theorem uses estimates in which 
the function~$u$ appearing inside $A_0$ in the norm~$\|\,\|_{\ell,A_0}$ from~\eqref{eq:a0norms} differs from the solution being
estimated. However, since in this paper the solutions being estimated are already known to exist
the function~$u$ appearing inside $A_0$ in the norms~\eqref{eq:a0norms} will simply be the solution
that is being estimated. 
\end{remark}
}

\begin{remark}
  The standard energy estimates for both symmetric hyperbolic systems without large terms and for
  singular limits obeying Klainerman-Majda balance involve only spatial derivatives of the
  solution. The reasons why time derivatives are also needed here and why an unweighted estimate for
  the time derivative is only obtained in the $L^2$ norm will be explained after the proof
  of Theorem~\ref{thm:unifexist}.
\end{remark}

\subsection{Assumptions and initial data}

The following standard conditions on the terms appearing in system~\eqref{eq:3pde} will be assumed,
where $s_0$ is defined in \eqref{eq:s0}:
\begin{assumption}\label{sum:hyp}
  \begin{enumerate}
  \item The matrices $A_0$ and the $A_j$ are symmetric and are $C^{s_0+1}$ functions of their
    arguments.

  \item The matrix $A_0$ is positive definite; more specifically there are positive 
constants~$c_0$
    and $b_0$ such that 
\begin{equation}\label{eq:c0b0}
\text{
$A_0(v)\ge c_0I$\quad for\quad $|v|\le b_0$.}
\end{equation}
\item The function $G(t,x)\eqdef F(t,x,0)$ is bounded in $H^{s_0+1}$ uniformly in~$t$, and for $1\le k\le
  s_0+1$ the $H^{s_0+1-k}$ norm of $\prt_t^kG$ is bounded uniformly in t. In addition, the
  function~$H(t,x,u)\eqdef \int_0^1 \ppp{F}{u}(t,x,\alp u)\,d\alp$ belongs to $C^{s_0+1}$.

  \item The operators $\mathcal L$ and $\mathcal M$ are anti-symmetric constant-coefficient
    differential or pseudodifferential operators of order at most one.
  \end{enumerate}
\end{assumption}

\begin{remark}\label{rem:FGH}
  The identity 
\begin{equation*}
F(t,x,u)-F(t,x,0)=\int_0^1 \dd{\alp} F(t,x,\alp u)\,d\alp=\lb\int_0^1
  \ppp{F}{u}(t,x,\alp u)\,d\alp\rb u
\end{equation*}
together with the definitions in Assumption~\ref{sum:hyp}
show
  that 
\begin{equation}\label{eq:FGH}
F(t,x,u)\equiv G(t,x)+H(t,x,u)u.
\end{equation}
\end{remark}

As noted in the introduction, the initial data will be required to
be chosen so that $u_t(0,x)$ from~\eqref{eq:ut0} is uniformly bounded in~$H^{s_0}$. From the
PDE~\eqref{eq:3pde} we see that 
this
 well-preparedness 
condition is
 equivalent to the condition that
\begin{equation}
  \label{eq:dlemu0}
  (\tfr\del\mathcal L+\tfr\eps\mathcal M)u(0,x,\del,\eps) \text{\ be uniformly bounded in $H^{s_0}$.}
\end{equation}

Under the above conditions
the following lemma shows that
the $\vertiv{\;}_{s_0+1,\eps,A_0}$ norm of $u$ will be
uniformly bounded at time zero. In the statements of both this result and the main theorem we will
use the Sobolev embedding constant, i.e., the constant~$K$ such that 
\begin{equation}\label{eq:K}
\sup_x|v(x)|\le K\|v\|_{{s_0}}.
\end{equation}
\begin{lemma}\label{lem:initbnd}
  Assume that initial data satisfy 
\begin{equation}\label{eq:initm}
\begin{aligned}
\|u_0(x,\del,\eps)\|_{{s_0+1}}\le m_1
\qquad\text{and}\qquad
\|(\tfr\del\mathcal L+\tfr\eps\mathcal M)u_0(x,\del,\eps)\|_{{s_0}}\le m_2
\end{aligned}
\end{equation}
for all 
\begin{equation}\label{eq:eps0c1c2}
\text{
$0<\eps\le \eps_0$\quad and\quad $0<c_1\eps^{1+\fr{s_0}}\le\del\le 1$,}
\end{equation}
that Assumption~\ref{sum:hyp} holds, and that $\eps_0Km_1\le \frac{b_0}2$, where $b_0$ is defined in
Assumption~\ref{sum:hyp}. 
{\clrr  Let $u$ be any function such that $u(0,x,\del,\eps)=u_0(x,\del,\eps)$,
  $u_t(0,x,\del,\eps)$ equals the right side of \eqref{eq:ut0} obtained by solving~\eqref{eq:3pde}
  for $u_t$ and setting $t$ equal to zero, and the higher time derivatives of
  $u$  at time zero through order~$s_0+1$ are determined recursively in similar fashion by solving 
  $\prt_t^j$ of the PDE~\eqref{eq:3pde}  for $\prt_t^{j+1}u$, setting $t$ equal to zero, and
  substituting in the values of lower time derivatives of $u$ at time zero already so determined.}

Then there is a constant~$M$ depending
only on the spatial dimension~$d$, the 
constants~$c_0$
from~\eqref{eq:c0b0}, $m_1$
and~$m_2$ from \eqref{eq:initm}, and $\eps_0$ and $c_1$ from~\eqref{eq:eps0c1c2},
the norms of $\mathcal L$ and $\mathcal M$ as operators from $H^1$ to $L^2$, the $C^{s_0+1}$
norms of 
$A_0$, $A_j$, and $H$ over the domain
$\{|u|\le 2Km_1\}$,
and the $H^{s_0+1}$ norm of $G$, 
such that {\clrr  at time zero}
\begin{equation}
  \label{eq:initM}
  (\vertiv{u}_{s_0+1,\eps,A_0})\eval{t=0}\le M
\end{equation}
for all $\del$ and $\eps$ satisfying the above conditions.
\end{lemma}

\begin{proof}
  Roughly speaking, the result of the lemma follows from the fact that when $u_t(0,x)$ is $O(1)$
  then 
using the PDE \eqref{eq:3pde} plus induction shows
that at time zero $\eps^k\prt_t^ku$ is
  $O(\frac{\eps^k}{\del^{k-1}})$ for $2\le k\le s_0+1$, which is $O(1)$ on account of the assumption
  that $\del\ge c_1\eps^{1+\fr{s_0}}$, and hence yields the uniform boundedness of the
  $\vertiv{\;}_{s_0+1,\eps,A_0}$ norm of $u$ at time zero. 

More specifically, 
by repeated applications of the PDE \eqref{eq:3pde} to express 
higher time derivatives in
  terms of 
$u$, $u_t$ and their spatial derivatives,
 and applications of $\mathcal L$ and $\mathcal
  M$ to them, we obtain 
that, for $2\le k\le s_0+1$, the
leading-order term of $\prt_t^k u$ is 
 $(\tfr{\del}\mathcal L+\tfr\eps\mathcal M)^{k-1}\prt_t u\eval{t=0}$,
 which yields
the estimate $\eps^k\|\prt_t^k u\eval{t=0}\|_{{s_0+1-k}}\le
  c\frac{\eps^k}{\del^{k-1}}$. To see this note first that the assumptions on the initial data
  ensure that $A_0\ge c_0I$. 
Applying 
 $\prt_t^{k-1}$ to~\eqref{eq:3pde}, using the invertibility of~$A_0$ to solve the result  for
  $\prt_t^ku$, taking up to $s_0+1-k$ spatial derivatives of the result, and summing the $L^2$ norms
  of the results yields a formula for the $H^{s_0+1-k}$ norm of $\prt_t^ku$ in terms of $L^2$ norms
  of products of spatial derivatives of lower-order time derivatives of~$u$. Note that coefficients
  such as $A_j(u)$ can be estimated in the maximum norm in terms of $\|u\|_{{s_0}}$ and so may be
  pulled out of those $L^2$ norms. 

For the case~$k=2$ this
  yields an estimate of $\|u_{tt}\|_{{s_0-1}}$ in terms of $L^2$ norms of products of the factors $u$, $u_x$
  $u_t$, $G$ and $G_t$ and their spatial derivatives of order at most $s_0$, with coefficients of
  size at most $O(\tfr\del)$ coming from the presence of $\tfr\del$ in the time derivative of \eqref{eq:3pde}.
Since all those factors are bounded in $H^{s_0}$ at time zero, and $H^{s_0}$ is an algebra, this
yields the estimate~$\|u_{tt}\eval{t=0}\|_{{s_0-1}}\le\frac c\del$.

The analogous expressions for $\|\prt_t^ku\|_{{s_0+1-k}}$ with $k>2$ include factors of $u_{tt}$
and possibly higher time derivatives, plus their spatial derivatives. Although~$u_{tt}$ and
higher-order time derivatives of $u$ do not belong to $H^{s_0}$ at time zero, the resulting
expressions could be estimated by the method used in the proof of Theorem~\ref{thm:unifexist} to
estimate similar expressions. However, it is simpler to use finite induction to express higher-order time
derivatives in terms of $u$ and $u_t$. Since the time derivative in~\eqref{eq:3pde} is expressed in
terms of expressions involving at most one spatial derivative, this again yields an estimate in
terms of  $L^2$ norms of products of the factors $u$, $u_x$,
  $u_t$ and their spatial derivatives of order at most $s_0$, 
plus time and spatial derivatives of $G$ of order at most $s_0$, this time with
  coefficients of size at most $O(\tfr{\del^{k-1}})$ since equation~\eqref{eq:3pde} is used at most
  $k-1$ times to express $k-1$ time derivatives in terms of spatial derivatives. Note that 
{\clrr   wherever $u_t$ and its spatial derivatives occur the time derivative is left unaltered rather
  than using}  \eqref{eq:3pde} to express $u_t$ in terms of $u$, because $u_t$ is $O(1)$ at time zero but the
  individual terms on the right side of \eqref{eq:3pde} may not be. This yields the estimates
  $\|\prt_t^ku\eval{t=0}\|_{{s_0+1-k}}\le\frac c{\del^{k-1}}$. As indicated at the beginning of
  the proof, these estimates together with assumption~\eqref{eq:eps0c1c2} show that the
  $\vertiv{\;}_{s_0+1,\eps,A_0}$ norm of $u$ is uniformly bounded at time zero.
\qed\end{proof}

The 
well-preparedness
condition \eqref{eq:dlemu0} can be achieved, for example,
 by using initial data of the form
\begin{equation}
  \label{eq:u0}
  u(0,x,\del,\eps)=u_0(x,\del,\eps)\eqdef\sum_{j=0}^m \lp\tfrac{\del}{\eps}\rp^{j} 
\widetilde u_{j}(x)+\del U_0(x,\del,\eps)
\end{equation}
for some nonnegative integer~$m$, with the $\widetilde u_{j}$ belonging to $H^{s_0+1}$ and $U_0$ 
bounded in that
space uniformly in~$\del$ and~$\eps$. 
In fact, since
\begin{equation*}
  (\tfr{\del}\mathcal L+\tfr\eps\mathcal M) u_0(x,\del,\eps)
=\tfr\del \mathcal L\widetilde u_0+\sum_{j=1}^m \tfrac{\del^{j-1}}{\eps^j}(\mathcal L\widetilde u_j+\mathcal
M\widetilde u_{j-1})+\tfrac{\del^m}{\eps^{m+1}}\mathcal M \widetilde u_m+O(1)
\end{equation*}
in view of the scaling assumption~\eqref{eq:delsmall},
condition~\eqref{eq:dlemu0} will hold provided that
\begin{equation*}
  \mathcal L\widetilde u_0=0,\;\; \mathcal L \widetilde u_j=-\mathcal M\widetilde u_{j-1} 
\text{\ for $j=1,\cdots,m$,}\;\;
\text{and either $\mathcal M\widetilde u_m=0$ or $\del^m\le c\eps^{m+1}$.}
\end{equation*}
For example, the well-preparedness condition holds when $m=0$ and $\mathcal L\widetilde
u_0=0=\mathcal M\widetilde u_0$, or when $m=s_0$, equality holds in \eqref{eq:deleps}, $\mathcal
L\widetilde u_0=0$, and 
\begin{equation}\label{eq:lujmujm1}
\mathcal L \widetilde u_j=-\mathcal M\widetilde u_{j-1} \text{\ for $j=1,\cdots,s_0$.}
\end{equation}
When the ranges of~$\mathcal L$ and~$\mathcal M$ overlap,
the condition~\eqref{eq:lujmujm1} allows more general initial data than would be
obtained by requiring that each
side of those equations vanish separately.

\subsection{Theorem and proof}

\begin{theorem}\label{thm:unifexist}
  Under the assumptions of Lemma~\ref{lem:initbnd}, there exists a constant~$T$ depending only on
  the quantities that $M$ in that lemma depends on, such that for all $\eps$ and $\del$
  satisfying~\eqref{eq:eps0c1c2} the solution of the initial-value problem~\eqref{eq:3pde},
  $u(0,x,\del,\eps)=u_0(x,\del,\eps)$ exists on $[0, T]$ and 
satisfies $\max_{0\le t\le T}\vertiv{u}_{s_0+1,\eps,A_0} \le 2M$.
\end{theorem}

\begin{proof}
{\clrr The local existence and continuation theorems (\cite[Ch 2., Theorems
2.1--2.2]{Majdabook84}) mentioned in Remark~\ref{rem:exist} ensure that the solution of the
initial-value problem exists on some time interval that might depend on $\del$ and $\eps$, and
will continue to exist for a time independent of those small parameters provided that it satisfies
an $H^{s_0+1}$ estimate independent of them. Hence it suffices to prove such an estimate. Moreover,
although the norm~$\vertiv{\,}_{s_0+1,\eps,A_0}$ used in the estimates below  depends on the
solution~$u$ being estimated, condition~\eqref{eq:c0b0} ensures that the resulting estimate will
indeed be uniform. 
}
The estimates that will be derived are similar to standard energy estimates for solutions of
symmetric hyperbolic systems but require keeping
track of the powers of $\del$ and $\eps$ that appear in those estimates for the system~\eqref{eq:3pde}. 

Applying $D^\alp\prt_t^k$ with $0\le k\le s_0+1$ and $0\le |\alp|\le s_0+1-k$ to \eqref{eq:3pde},
taking the inner product with $2D^\alp\prt_t^k u$, integrating over the spatial variables,
integrating by parts in the terms that involve $A_j$ undifferentiated, noting that the terms
involving $\mathcal L$ or $\mathcal M$ drop out on account of the anti-symmetry of those operators,
summing over all $\alp$ satisfying the above-mentioned condition, 
and multiplying the result by the weight~$\eps^{2k}$ yields
  \begin{equation}
    \label{eq:est1}
\begin{aligned}
    &\tddt \big[\eps^{2k}\|\prt_t^k u\|_{s_0+1-k,A_0}^2\big]=\tddt\lb\eps^{2k}
\sum_{0\le|\alp|\le s_0+1-k}\int (D^\alp\prt_t^k u)\cdot A_0(\eps u)(D^\alp\prt_t^k u)\,dx\rb
\\&=\eps^{2k}\sum_{ 0\le|\alp|\le s_0+1-k}\int (D^\alp\prt_t^k u)\cdot
\Big[ \eps u_t\cdot \gradu A_0+\sum_j u_{x_j}\cdot\gradu A_j \Big]
(D^\alp\prt_t^k u)\,dx
\\&\fakeq+2\eps^{2k}\cdot
\\[-4pt]&\;\;\sum_{0\le|\alp|\le s_0+1-k}\int (D^\alp\prt_t^k u)\cdot
\lB D^\alp\prt_t^k (G+Hu) -  [D^\alp\prt_t^k,A_0]u_t-\sum_j[D^\alp\prt_t^k,A_j]u_{x_j}\rB\,dx
\\&\le
\|\eps u_t\cdot \gradu A_0+\sum_j u_{x_j}\cdot\gradu A_j\|_{L^\infty}
\eps^{2k}\|\prt_t^ku\|_{s_0+1-k}^2
\\&\fakeq+c {\eps^{k}}\|\prt_t^k u\|_{s_0+1-k}\cdot
\\&\fakeq\;\;\;\Big({\eps^{k}}\|\prt_t^k G\|_{s_0+1-k}+{\eps^{k}}\|H\|_{L^\infty}\|\prt_t^ku\|_{s_0+1-k}
+{\eps^{k}}\Big[\sum_{0\le|\alp|\le s_0+1-k}\!\!\!\!\!\!\|[D^\alp\prt_t^k,A_0]u_t\|_{L^2}^2\Big]^{1/2} 
\\&\fakeq\fakeq\fakeq\fakeq\fakeq\fakeq\fakeq
+{\eps^{k}}\Big[\sum_{0\le|\alp|\le s_0+1-k}\big(\|D^\alp\prt_t^k,H]u\|_{L^2}^2
+\sum_j\|[D^\alp\prt_t^k,A_j]u_{x_j}\|_{L^2}^2\big)\Big]^{1/2} \Big),
\end{aligned}
\end{equation} 
where the inequality is obtained by pulling out $\eps u_t\cdot \gradu A_0+\sum_j u_{x_j}\cdot\gradu
A_j$ from the first integral in maximum norm, and breaking the second integral into several parts and
using the Cauchy-Schwartz inequality in each of them.

Since $A_0=A_0(\eps u)$ will be differentiated at least once when it appears in any
commutator term on the right side of the inequality in~\eqref{eq:est1}, which yields at least
  one power of $\eps$, the power of $\eps$ in
every term appearing on the right side of the inequality in~\eqref{eq:est1} is at least as large as
the total number of time derivatives in that term.  By the definition of the
$\vertiii{\;}_{s_0+1,\eps,A_0}$ norm plus the smoothness assumption on~$A_0$, this implies that in
order to bound the right side of the inequality in~\eqref{eq:est1} by a continuous function
of~$\vertiii{u}_{s_0+1,\eps,A_0}$ it suffices to bound all the terms there by a continuous function
of $\vertiii{u}_{s_0+1,1}$ after replacing $\eps$ by~$1$ and replacing $A_0$ by the identity
  matrix.

The condition on $s_0$ ensures that $\|u_t\|_{L^\infty}$ and $\|\grad u\|_{L^\infty}$ are bounded by
a constant times $\|u_t\|_{s_0}$ and $\|u\|_{s_0+1}$, respectively, and those norms are each bounded
by $\vertiii{u}_{s_0+1,1}$. By the smoothness of the $A_j$,
$\sum_{j=0}^d\|\grad_u A_j\|_{L^\infty}\le c(\|u\|_{s_0})\le \widetilde c(\vertiii{u}_{s_0+1,1})$ for
some continuous function~$\widetilde c$. This yields the desired estimate for the entire first term
on the right side of the inequality in~\eqref{eq:est1}. The terms on the right side of the
inequality in~\eqref{eq:est1} in which $G$ and the $L^\infty$ norm of $H$ appear are also so bounded in
view of the assumptions of those functions. 

There remains to estimate only the terms on the right side of the inequality in~\eqref{eq:est1} that
involve commutators. Since the factor $\|\prt_t^k u\|_{s_0+1-k}$ multiplying the norms of the
commutators is one of the terms in $\vertiii{u}_{s_0+1,1}$, only the norms of the commutator terms
themselves must be estimated.  We can pull out in the
$L^\infty$ norm any factor such as $\grad_u H$ that depends only on $t$, $x$ and $u$ without
derivatives, and the assumptions on the various coefficients ensure that each factor so pulled out
is bounded by a continuous function of $\|u\|_{s_0}$ and hence by a continuous function of
$\vertiii{u}_{s_0+1,1}$. Since the presence of the commutator ensures that at least one
derivative will be applied to the function appearing in the commutator, the 
terms arising from the commutators that remain inside the $L^2$ norms
all take the form 
\begin{equation}\label{eq:prodder}
\Big[\int \prod_{\ell=1}^L |D^{\alp_\ell}\prt_t^{k_\ell}u|^2\,dx\Big]^{1/2},
\end{equation}
where $L\ge2$, $1\le |\alp_\ell|+k_\ell\le s_0+1$,  and $\sum_\ell (|\alp_\ell|+k_\ell)\le s_0+2$. 
If $|\alp_\ell|+k_\ell=s_0+1$ for some~$\ell$ then only one derivative is applied to the other
factor, so that factor can be pulled out in $L^\infty$ norm and estimated by $\|u\|_{s_0+1}$ or
$\|u_t\|_{s_0}$, both of which appear in $\vertiii{u}_{s_0+1,1}$. After pulling out that factor the
integral becomes  $\int
|D^{\alp_\ell}\prt_t^{k_\ell}u|^2\,dx$, which is bounded by $\|\prt_t^{k_\ell}u\|_{s_0+1-k_\ell}^2$,
which also appears in $\vertiii{u}_{s_0+1,1}$. Otherwise $|\alp_\ell|+k_\ell\le s_0$ for all $\ell$,
and by using the multiple-factor version of H\"older's inequality
we will bound the integral in 
\eqref{eq:prodder} by
\begin{equation}\label{eq:terms2}
  \prod_{\ell=1}^L \lp\int \la D^{\alp_\ell}\prt_t^{k_\ell}u\ra^{2p_\ell}\,dx\rp^{\fr{p_\ell}},
\end{equation}
where the exponents $p_\ell$, which will be chosen later, 
must satisfy  
\begin{equation}\label{eq:phold}
1\le p_\ell\le\infty \quad\text{and}\quad \sum_\ell \frac 1{p_\ell}=1.
\end{equation}
The integrals in ~\eqref{eq:terms2} will then be bounded via
the Gagliardo-Nirenberg inequality (e.g., \cite[p. 24]{friedman76:pde})
\begin{equation}
  \label{eq:gn}
  \|v\|_{L^p}\le c\| v\|_{r}^a\|v\|_{L^2}^{1-a},
\end{equation}
in which the parameters must satisfy $\tfr p=\tfr2-\frac{ar}d$, $r\ge1$, and $0\le a< 1$, where as usual $d$
is the spatial dimension. Although $a$ is actually allowed to equal the endpoint value~$1$ for many
values of the other parameters, we will avoid that value in order to obtain a unified proof.
The inequality constraint on~$a$ will hold provided that $\half\ge \fr
p>\fr2-\frac rd$.  In order to estimate the integrals in~\eqref{eq:terms2} we apply \eqref{eq:gn}
with $v\eqdef D^{\alp_\ell}\prt_t^{k_\ell}u$, so we will let $r=s_0+1-(|\alp_\ell|+k_\ell)$, since
that is the highest Sobolev norm of $D^{\alp_\ell}\prt_t^{k_\ell}u$ that is bounded by
$\vertiii{u}_{s_0+1,1}$. Since we only use \eqref{eq:terms2} when $|\alp_\ell|+k_\ell\le s_0$ for
all $\ell$, the condition $r\ge1$ will indeed hold. Since the norm of $D^{\alp_\ell}\prt_t^{k_\ell}u$ appearing in
\eqref{eq:terms2} is the $L^{2p_\ell}$ norm, $p$ in \eqref{eq:gn} equals~$2p_\ell$. Substituting in
these values  and multiplying
everywhere by two turns the inequality constraint on $p$ into the inequality constraint
\begin{equation}
  \label{eq:pineq}
  1 \ge \fr{p_\ell}>1-\frac{2(s_0+1-(|\alp_\ell|+k_\ell))}d
\end{equation}
on $p_\ell$. We now show that it is possible to choose the $p_\ell$ such that both \eqref{eq:phold}
and \eqref{eq:pineq} hold.

Since $|\alp_\ell|+k_\ell\le s_0$, the interval to which $\fr{2p_\ell}$ is restricted by \eqref{eq:pineq} is nonempty.
Since $|\alp_\ell|+k_\ell\ge 1$ and $\frac d2\le s_0\le \frac d2+1$, the lower limit in \eqref{eq:pineq} is negative
iff $|\alp_\ell|+k_\ell=1$. The inequality in \eqref{eq:phold} is equivalent to 
\begin{equation}
  \label{eq:pholdineq}
  1\ge \fr{p_\ell}\ge 0. 
\end{equation}
Combining \eqref{eq:pholdineq} with \eqref{eq:pineq} yields
\begin{equation}
  \label{eq:pineq2}
  1\ge \fr{p_\ell}>\max\lp 0,1-\frac{2(s_0+1-(|\alp_\ell|+k_\ell))}d\rp
\end{equation}
where for simplicity we ignore the possibility $p_\ell=\infty$, which will not be needed.
Every value of $p_\ell$ satisfying \eqref{eq:pineq2} is allowed by both \eqref{eq:pineq} and
\eqref{eq:pholdineq}, so it suffices to show that we can choose values in the intervals in
\eqref{eq:pineq2} that sum to one. That is possible iff the sum of the lower values there is less
than one and the sum of the upper values is at least one. The latter condition holds
trivially, and as noted above the second expression inside the $\max$ in \eqref{eq:pineq2} is negative iff
$|\alp_\ell|+k_\ell=1$, so it suffices to show that
\begin{equation}
  \label{eq:sumineq}
  1>\sum_{\substack{1\le \ell\le L\\ |\alp_\ell|+k_\ell\ge 2}} \lb
1-\frac{2(s_0+1-(|\alp_\ell|+k_\ell))}d\rb
\end{equation}
Let $L_2$ denote the number of values of $\ell$ for which $|\alp_\ell|+k_\ell\ge2$. If $L_2=0$ then
the sum on the right side of \eqref{eq:sumineq} vanishes, so that condition indeed holds. When
$L_2\ge1$ then \eqref{eq:sumineq} can be written more explicitly as
\begin{equation}
  \label{eq:sumineq2}
  1>L_2(1-\tfrac 2d(s_0+1))+\tfrac 2d\lb \lp \sum_{\ell=1}^L(|\alp_\ell|+k_\ell)\rp-(L-L_2)\rb.
\end{equation}
Condition \eqref{eq:sumineq2} can be rewritten as
\begin{equation}
  \label{eq:sumineq3}
  (L_2-1)(s_0-\tfrac d2)>(2-L)+\Big(\sum_{\ell=1}^L(|\alp_\ell|+k_\ell)-(s_0+2)\Big).
\end{equation}
Since $L_2\ge1$ by assumption, $s_0>\frac d2$, $L\ge2$ and $\sum_{\ell=1}^L(|\alp_\ell|+k_\ell)\le s_0+2$,
the left side of \eqref{eq:sumineq3} is non-negative, and the right side there is non-positive.
Moreover, since $L\ge2$, if $L_2=1$ then exists an $\ell$ for which $|\alp_\ell|+k_\ell=1$, and in that case
the fact that $|\alp_\ell|+k_\ell\le s_0$ implies that either
either $L>2$ or $\sum_{\ell}(|\alp_\ell|+k_\ell)<s_0+2$. This shows that either the left side of
\eqref{eq:sumineq3} is strictly positive or the right side there is strictly negative, and hence
that inequality indeed holds. 

Summing over $0\le k\le s_0+1$ the estimates that we have obtained shows that
\begin{equation}
  \label{eq:vert3est}
  \ddt\vertiii{u}_{s_0+1,\eps,A_0}^2\le c(\vertiii{u}_{s_0+1,\eps,A_0})
\end{equation}
for some continuous function~$c$.

Finally, differentiating~\eqref{eq:3pde} with respect to $t$, 
taking the inner product of the result with $2u_t$, integrating over the spatial variables,
integrating by parts in the terms that involve $A_j$ undifferentiated, and noting that the terms
involving $\mathcal L$ or $\mathcal M$ drop out on account of the anti-symmetry of those operators
yields
  \begin{equation}
    \label{eq:estut}
\begin{aligned}
    &\tddt \big[\|u_t\|_{0,A_0}^2\big]=\tddt\lb
\int u_t\cdot A_0(\eps u) u_t \,dx\rb
\\&=\int \prt_t u\cdot
\Big[ \eps u_t\cdot \gradu A_0+\sum_j u_{x_j}\cdot\gradu A_j \Big]u_t
\,dx
\\[-4pt]&\fakeq+2\int u_t\cdot
\lB \prt_t (G+Hu) -  (\eps u_t\cdot \grad_u A_0)u_t-\sum_j (u_t\cdot\grad_u A_j)u_{x_j}\rB\,dx
\\&\le
\|\eps u_t\cdot \gradu A_0+\sum_j u_{x_j}\cdot\gradu A_j\|_{L^\infty}
\|\prt_tu\|_0^2
\\&\fakeq+c \|\prt_t u\|_{0}
\Big(\|\prt_t G\|_{0}+\|H\|_{L^\infty}\|\prt_tu\|_{0}
+\|\tppp{H}{t}\|_{L^\infty}\|u\|_0+\|u_t\|_0\|\grad_u H\|_{L^\infty}\|u\|_{L^\infty}
\\&\fakeq\fakeq\fakeq\fakeq\fakeq\fakeq\fakeq\fakeq
+\eps \|u_t\|_{L^\infty}\|\grad_u A_0\|_{L^\infty}\|u_t\|_0
+\|u_t\|_0 \sum_j\|\grad_u A_j\|_{L^\infty}\|u_{x_j}\|_{L^\infty}\Big),
\\&\le
c(\vertiii{u}_{s_0+1,\eps,A_0})\lb \|u_t\|_0^2+\|u_t\|_0\rb
\end{aligned}
\end{equation} 
where the first inequality follows in similar fashion to~\eqref{eq:est1} and the second from 
Assumption~\ref{sum:hyp} plus the
definition of the $\vertiii{\;}$ norm. Adding \eqref{eq:estut} to \eqref{eq:vert3est} yields the
uniform estimate
\begin{equation}
  \label{eq:vert4est}
   \ddt\vertiv{u}_{s_0+1,\eps,A_0}^2\le c(\vertiv{u}_{s_0+1,\eps,A_0})
\end{equation}
for some continuous function~$c$. 
By Lemma~\ref{lem:initbnd},  $\vertiv{u}_{s_0+1,\eps,A_0}$ is bounded uniformly in
$\del$ and $\eps$ by $M$ at time zero, so the differential inequality \eqref{eq:estut}
shows that there is a fixed positive
constant~$T$ such that $\max_{0\le t\le T}\vertiv{u}_{s_0+1,\eps,A_0}\le 2M$. 
\qed\end{proof}

\begin{remark}
  \begin{enumerate}

\item In the standard energy estimates for spatial derivatives of solutions of systems without large
  terms and of systems satisfying Klainerman-Majda balance, integrals of the
  form~\eqref{eq:terms2} not containing time derivatives are estimated using
 the Gagliardo-Nirenberg inequality 
\begin{equation}\label{eq:gninf}
 \|D^\alp v\|_{L^p}
\le c \|v\|_{s}^{a} \|v\|_{L^\infty}^{1-a}
\end{equation}
with $p=\frac{2s}{|\alp|}$, $a=\frac{|\alp|}{s}$, $s\ge s_0$, and $s>|\alp|$ instead of
\eqref{eq:gn}.  However, it is not possible to use \eqref{eq:gninf} to estimate integrals involving
second and higher time derivatives, because the boundedness of $\vertiii{u}_{s_0+1,\eps,A_0}$ does
not imply even an $\eps$-dependent bound for $\|\prt_t^k u\|_{L^\infty}$ when $k\ge2$.

  \item The special case of \eqref{eq:terms2} and \eqref{eq:gn} in which $p_\ell=2$,
    $|\alp_\ell|+k_\ell=2$, $p=4$, $r=1$, $a=\frac d4$, and $d$ is either two or three so that
    $s_0=2$ was used previously in \cite[\S 4.1 and Appendix]{mr3284568}. 

  \item The expression $\eps u_t$ appears in the estimates for a purely spatial derivative $D^\alp$
    of $u$, arising from the commutator term $[D^\alp,A_0]u_t$.  When the spatial derivative terms
    in the PDE are at most $O(\fr\eps)$ then substituting for $\eps u_t$ from the PDE yields a
    spatial derivative term of order one. Making this substitution allows spatial derivatives to be
    estimated without requiring estimates of time derivatives, both for systems without large terms
    and in the Klainerman-Majda theory.  However, for the PDE \eqref{eq:3pde} with the scaling
    \eqref{eq:delsmall} this procedure cannot be used because it yields terms of order
    $\frac\eps\del$, which is large. It is therefore necessary to leave the term $\eps u_t$ on the
    right side of the energy estimates for spatial derivatives of $u$, and this necessitates
    estimating time derivatives as well. Similarly, two-scale systems for which $A_0$ depends on $u$
    rather than $\eps u$ also require estimates of time derivatives
    \cite{mr665380,schochet88:asympt}.

  \item In similar fashion, a term containing $\eps u_{tt}$ appears in estimates for a spatial
    derivative of $u_t$ on account of the commutator term $[D^\alp\prt_t,A_0]u_t$. Assuming that
    $u_t$ is bounded initially but $\eps u_{tt}$ is large at time zero, this prevents us from
    obtaining an unweighted estimate for spatial derivatives of $u_t$. The reason we do obtain an
    unweighted estimate for $u_t$ itself is that the commutator term $[\prt_t, A_0]u_t$ does not
    yield any second time derivative.

\item The bound \eqref{eq:eps0c1c2} on how fast $\del$ can tend to zero compared to $\eps$ is only
  needed to ensure that the $\vertiv{\;}_{s_0+1,\eps,A_0}$ norm of the solution is uniformly bounded
  at time zero. The proof of Theorem~\ref{thm:unifexist} therefore also yields uniform bounds for a
  uniform time in the case when the time derivatives of the solution through order $s_0+1$ are
  uniformly bounded at time zero, without the need for assumption \eqref{eq:eps0c1c2} and without
  using weights of powers of $\eps$ in the norms. In particular, taking $\eps\equiv1$ and letting
  $\del\to 0$ yields a proof
  for arbitrary dimensions of the uniform existence theorem stated in \cite{mr665380} but only
  proven there
  in the case $d=1$, for which no Gagliardo-Nirenberg estimates are needed.
  \end{enumerate}
\end{remark}

\section{Convergence}
\subsection{A finite-dimensional perturbation result}
We begin with a result on perturbations of {\clrr  self-adjoint} 
matrices {\clrr ${\mathcal T}(\mu)\eqdef\fr{\mu^p}(T^{(0,0)}+\mu
T^{(0,1)})$, where $\mu$ is a small parameter, and $p$ is a
positive integer.} The result will be used in the proof of the
convergence theorem in Subsection \ref{subs:mainTheorem} ,
where $T^{(0,0)}$ and $ T^{(0,1)}$ will stand for the Fourier symbols of operators $\mathcal L$ and
$\mathcal M$ respectively.
{\clrr 
The result says that there is an orthogonal projection~$\mathcal P(\mu)$ 
that commutes with $\mathcal T(\mu)$, on whose range $\mathcal T(\mu)$ is bounded
uniformly and has a limit as $\mu\to0$, and on whose null space $\mathcal T(\mu)$ is bounded from
below by a constant times $\fr\mu$ and has a finite expansion in inverse powers of $\mu$.
}
\begin{lemma}\label{lem:pert}
Define $\mathcal T(\mu)\eqdef \tfr{\mu^p}T^{(0)}(\mu)\eqdef \tfr{\mu^p}(T^{(0,0)}+\mu T^{(0,1)})$, where
$T^{(0,0)}$ and $T^{(0,1)}$ are operators on a finite dimensional inner-product space~$X$ that are
either both self-adjoint or both skew-adjoint, $\mu$ is a small parameter, and $p$ is a positive integer.
Then
  \begin{enumerate}
  \item There exists an orthogonal projection operator $\mathcal P(\mu)$ that commutes with $\mathcal T(\mu)$
    for $\mu\ne0$, 
   is analytic in $\mu$ for real
    $\mu$, and satisfies
    \begin{align}
      \label{eq:projbnd}
      \|\mathcal P(\mu)\mathcal T(\mu)\mathcal P(\mu) f\|_{X}&\le c_1\|f\|_{X}
\displaybreak[0]\\
\label{eq:projlarge}
  \|(I-\mathcal P(\mu))\mathcal T(\mu)(I-\mathcal P(\mu))f\|_{X}&\ge \tfrac{c_2}{\mu}\|(I-\mathcal P(\mu))f\|_{X}
    \end{align}
for $0<\mu<\mu_0$, where $\|\,\|_{X}$ is the norm on the space~$X$ and $\mu_0$ and the $c_j$ are positive constants.

\item For $0\le j\le p{\clrr -1}$ there exist commuting orthogonal projection operators $P^{(j)}$ 
such that the ranges of the complementary projections $I-P^{(j)}$ are mutually orthogonal subspaces, and
  \begin{equation}
    \label{eq:pform}
    \mathcal P(0)=
\prod_{j=0}^{p-1} P^{(j)}=I-\sum_{j=0}^{p-1} (I-P^{(j)}),
\end{equation}

\item The $P^{(j)}$ are the orthogonal projection operators onto the null spaces of
  operators~$T^{(j,j)}$, which are 
 are determined from $T^{(0)}(\mu)\eqdef T^{(0,0)}+\mu T^{(0,1)}$ via the
 reduction process of \cite[\S II.2.3]{MR678094}: Specifically, after modifying the notation to
 facilitate repeated re{\clrr d}uctions, the $T^{(j,j)}$ are the first term{\clrr s} in the expansions
  \begin{equation}
    \label{eq:tj}
    T^{(j+1)}(\mu)\eqdef \fr{\mu}\widetilde P^{(j)}(\mu)T^{(j)}(\mu)\widetilde P^{(j)}(\mu)=\sum_{k=0}^\infty \mu^k T^{(j+1,j+1+k)}.
  \end{equation}
Here $\widetilde
P^{(j)}(\mu)=\widetilde P^{(j-1)}(\mu) P^{(j)}(\mu)\widetilde P^{(j-1)}(\mu)$ for $j\ge0$,
$P^{(j)}(\mu)$ is the orthogonal projection onto the direct sum of the eigenspaces of $T^{(j)}(\mu)$ of all
eigenvalues of order $o(1)$, and 
$\widetilde P^{(-1)}(\mu)\eqdef I$. The $T^{(j,k)}$ are all self-adjoint when $T^{(0,0)}$ and
$T^{(0,1)}$ are self-adjoint, and are all skew-adjoint when $T^{(0,0)}$ and
$T^{(0,1)}$ are skew-adjoint.

\item
\begin{equation}
\label{eq:limpt}
\lim_{\mu\to0} \mathcal P(\mu)\mathcal T(\mu)
=  T^{(p,p)}.
\end{equation}

\item The operator $T^{(0,0)}$ is given, and for $1\le j\le 2$,
  \begin{align}
    \label{eq:t1}
T^{(1,1)}&= P^{(0)}T^{(0,1)} P^{(0)}
\\
\label{eq:t2}
T^{(2,2)}&=-\widetilde P^{(1)} T^{(0,1)}
\lp T^{(0,0)}\rp^{-1}_{\pseudo}
T^{(0,1)} \widetilde P^{(1)}
\end{align}
where $\widetilde P^{(j)}\eqdef\prod_{\ell=0}^j P^{(j)}$, and
$(M)^{-1}_\pseudo$ denotes the pseudo-inverse of the matrix~$M$, defined by
$\lp C^{-1}\lp\begin{smallmatrix}M_{11}&0\\0&0\end{smallmatrix}\rp C\rp^{-1}_\pseudo\eqdef 
C^{-1}\lp\begin{smallmatrix}M_{11}^{-1}&0\\0&0\end{smallmatrix}\rp C$.

  \end{enumerate}
\end{lemma}

\begin{proof}
  If both $T^{(0,0)}$ and $T^{(0,1)}$ are skew-adjoint then multiplying both of them by $i$ makes
  them self-adjoint without affecting the projections, so we may assume that they are
  self-adjoint. Moreover, as noted in \cite[\S II.6.1]{MR678094}, the reduction process preserves
  self-adjointness and so may be continued without limitation, since the nilpotent factors of the
  general case are absent. In particular, by \cite[Theorem 6.1 in \S II.6.1]{MR678094} the
  eigenvalues of $T^{(j)}(\mu)$ and the projection operators $P^{(j)}(\mu)$ are all analytic for
  real~$\mu$. As in \cite[\S II.1.3]{MR678094} let $R(z,\mu)$ denote the resolvent
  $(T^{(0)}(\mu)-z)^{-1}$ wherever $z$ is not an eigenvalue of $T^{(0)}(\mu)$.  In view of the
  analyticity of the eigenvalues, \cite[(1.16) in \S II.1.4]{MR678094} shows that for $\mu$
  sufficiently small the operator 
\begin{equation*}
\mathcal P(\mu)=-\tfr{2\pi
    i}\int_{|z|=\mu^{p-\half}}R(z,\mu)\,dz
\end{equation*}
is the orthogonal projection onto the direct sum of the
  eigenspaces of the eigenvalues of $\mathcal T(\mu)$ of size at most $O(1)$, and $I-\mathcal
  P(\mu)$ is the orthogonal projection onto the direct sum of the eigenspaces of the eigenvalues of
  $\mathcal T(\mu)$ of size at least $O(\mu^{-1})\gg1$. These estimates show that
  \eqref{eq:projbnd}\textendash\eqref{eq:projlarge} hold.

We carry out the reduction process of \cite[\S II.2.3]{MR678094} while choosing the unperturbed
eigenvalue zero at every stage. However, we do not want to include the range of $I-P^{(j-1)}$ when
considering the zero eigenspace of $T^{(j,j)}$ since that subspace has already been accounted for at
previous stages of the reduction process. For this reason we replace the factor $P^{(j)}(\mu)$,
which would appear in \eqref{eq:tj} if the corresponding formula \cite[(2.37) in \S
II.2.3]{MR678094} were simply rewritten in our notation, with $\widetilde P^{(j)}(\mu)$. This
corresponds to the suggestion in \cite[\S II.2.3]{MR678094} to add a constant multiple of
$I-P^{(j-1)}$ to $T^{(j,j)}$ but without the need to modify that operator. This procedure yields \eqref{eq:tj}. 

Since $\mathcal T(\mu)$, after multiplication by $i$ if necessary, is self-adjoint for all $\mu$, formula
\eqref{eq:tj} implies that all the $T^{(j,k)}$ are then also self-adjoint after that multiplication
has been done if necessary.
If at any stage of the reduction process the new unperturbed
operator $T^{(j+1,j+1)}$ does not have any zero eigenspace except for the range of $(I-P^{(j)})$, then
$P^{(j+1)}$ and hence also $\mathcal P(0)$ is identically zero, and if $j+1<p-1$ then $T^{(k)}(\mu)$ is
simply the zero operator for $j+1<k\le p-1$.

By the construction of the reduction process, $I-P^{(j)}(\mu)$ is the orthogonal projection onto the
direct sum of the eigenspaces of $\mathcal T(\mu)$ whose eigenvalues are of size
$O(\mu^{j-p})$. Since the eigenvalues for different values of $j$ are distinct for small enough
$\mu$, and $\mathcal T(\mu)$ is self-adjoint, the ranges of
$I-P^{(j)}(\mu)$ for different values of $j$ are orthogonal to
each other. This implies that the $I-P^{(j)}(\mu)$ for different $j$ commute with each other, and hence
so do the $P^{(j)}(\mu)$. Since $I-\mathcal P(\mu)$ is the orthogonal projection onto the union over $0\le j\le
p-1$ of the eigenspaces of $\mathcal T(\mu)$ whose eigenvalues are of size
$O(\mu^{j-p})$, and those eigenspaces are orthogonal for distinct $j$, 
\begin{equation}\label{eq:imp}
I-\mathcal P(\mu)=\sum_{j=0}^{p-1} (I-P^{(j)}(\mu)).
\end{equation}
In addition, since the $(I-P^{(j)}(\mu))$ project onto mutually orthogonal subspaces,
$(I-P^{(j_1)}(\mu))(I-P^{(j_2)}(\mu))=0$ for $j_1\ne j_2$, which implies that
\begin{equation}
  \label{eq:prodsum}
\begin{aligned}
  \prod_{j=0}^{p-1} &P^{(j)}(\mu)=\prod_{j=0}^{p-1}(I-(I-P^{(j)}(\mu))
\\& =I-\sum_{j=0}^{p-1}  (I-P^{(j)}(\mu))+\text{terms with at least two distinct factors $(I-P^{(j_k)}(\mu))$}
\\&=I-\sum_{j=0}^{p-1} (I-P^{(j)}(\mu)).
\end{aligned}
\end{equation}
Since the eigenvalues of size $O(1)$ of $T^{(j)}(\mu)$ are the perturbations of the nonzero
eigenvalues of $T^{(j,j)}$, the continuity of the projections $P^{(j)}(\mu)$ shows that as $\mu$
tends to zero the orthogonal projection $I-P^{(j)}(\mu)$ onto the direct sum of the eigenspaces of
eigenvalues of $T^{(j)}(\mu)$ that are $O(1)$ tends to the orthogonal projection $I-P^{(j)}$ onto
the direct sum of the eigenspaces of eigenvalues of $T^{(j,j)}$ that are nonzero. This shows
that 
\begin{equation}\label{eq:pjmutopj}
 P^{(j)}(\mu)\to P^{(j)}\quad \text{as}\quad \mu\to0 
\end{equation}
in the strong (finite-dimensional) operator topology, 
which is isometric to a suitably normed matrix space. Therefore,  taking the limit of
\eqref{eq:imp}\textendash\eqref{eq:prodsum} and rearranging yields \eqref{eq:pform}.  Taking the
limit of the identities $P^{(j)}(\mu)P^{(k)}(\mu)=P^{(k)}(\mu)P^{(j)}(\mu)$ yields
$P^{(j)}P^{(k)}=P^{(k)}P^{(j)}$, and the orthogonality of the ranges of $I-P^{(j)}(\mu)$ imply the
orthogonality of the ranges of $I-P^{(j)}$. This also shows that $P^{(j)}$ is the orthogonal
projection onto the null space of $T^{(j,j)}$, where the $T^{(j,j)}$ are the first terms in the
expansions~\eqref{eq:tj}.

Since $\mathcal P(\mu)$ is the orthogonal projection onto the direct sum of the eigenspaces of
$\mathcal T(\mu)$ that are $O(1)$ or $o(1)$, continuing the reduction process one more step
yields  \eqref{eq:limpt}.

The formulas for the $T^{(j,j)}$ are obtained by using recursively formula \cite[(2.18)  
in \S II.2.2]{MR678094}, which in our notation  becomes, for the case here in which there are no nilpotents,
\begin{gather}\label{eq:tjj}
  T^{(j+1,j+n)}=-\sum_{r=1}^{n} (-1)^r
  \sum_{\substack{\sum_{\ell=1}^r\nu_\ell=n\\\sum_{\ell=1}^{r+1}k_\ell=r-1\\\nu_\ell\ge1, k_\ell\ge0}}
  S^{(j,k_1)}T^{(j,j+\nu_1)}S^{(j,k_2)}\cdots S^{(j,k_r)}T^{(j,j+\nu_r)}S^{(j,k_r+1)},
\\
\label{eq:sform}
S^{(j,0)}\eqdef -P^{(j)},\quad S^{(j,\ell)}\eqdef
\lp \lp T^{(j,j)}\rp^{-1}_\pseudo\rp^\ell \text{ for $\ell\ge1$.}
\end{gather}
In particular, for $j=0$ and $n=1$ only the term with $r=1$ is present in the outer sum in \eqref{eq:tjj},
and the inner sum then contains only the case where $\nu_1=1$ and $k_1=0=k_2$. Using
\eqref{eq:sform}, this yields \eqref{eq:t1}. An analogous but longer calculation yields \eqref{eq:t2}.
\qed\end{proof}

\begin{remark}
\begin{enumerate}
\item Although $T^{(0,k)}\equiv0$ for $k\ge2$, $T^{(j,k)}$ may be nonzero for arbitrarily large
  values of $k$ when $j\ge 1$.
\item Formula \eqref{eq:tjj} shows that in order to calculate $T^{(2,2)}$ 
it is necessary to first calculate $T^{(1,1)}$ and $T^{(1,2)}$, while in
  order to calculate $T^{(3,3)}$ it would be necessary to first calculate $T^{(1,j)}$ for $1\le j\le
  3$ and then
  $T^{(2,j)}$ for $2\le j\le 3$. 
\end{enumerate}
\end{remark}

\begin{example}\label{eg:tjj}
  \begin{enumerate}
   \item\label{item:1}   
In the application of Lemma~\ref{lem:pert} to the convergence theorem the operators $T^{(0,0)}$
     and $T^{(0,1)}$ will be individual Fourier modes of the operators $\mathcal L$ and $\mathcal M$
     from \eqref{eq:3pde}. For example, if 
$\mathcal L=\lp\begin{smallmatrix}\prt_x&0\\0&0
     \end{smallmatrix}\rp$ and $\mathcal M=\lp\begin{smallmatrix}0&\prt_y\\ \prt_y&0
     \end{smallmatrix}\rp$ then $T^{(0,0)}=\lp\begin{smallmatrix}i k&0\\0&0
     \end{smallmatrix}\rp$ and $T^{(0,1)}=\lp\begin{smallmatrix}0&i\ell \\ i\ell&0
     \end{smallmatrix}\rp$ for some fixed values of $k$ and $\ell$. When $k\ne0$ 
     the projection onto the null space of $T^{(0,0)}$ is 
   $P^{(0)}=\lp\begin{smallmatrix}0&0\\0&1 \end{smallmatrix}\rp$, and
 formulas \eqref{eq:t1} and \eqref{eq:t2} yield $T^{(1,1)}=P^{(0)}T^{(0,1)}P^{(0)}=0$ and
$T^{(2,2)}= \lp\begin{smallmatrix}0&0\\0&\frac{-i \ell^2}k
     \end{smallmatrix}\rp$ since $P^{(1)}=I$ so $\widetilde P^{(1)}=P^{(0)}P^{(1)}=P^{(0)}$. If
     $\ell\ne0$ then
   $T^{(0,0)}$ and $T^{(2,2)}$ each have one nonzero eigenvalue so the fact that the matrices
     are of size $2\times2$ implies that $T^{(j,j)}=0$ for $j>2$, while if $\ell=0$ then
     $T^{(j,j)}=0$ for $j\ge 2$. When $k=0$
     but $\ell$ is nonzero then $T^{(0,0)}=0$, $P^{(0)}=I$, $T^{(1,1)}=T^{(0,1)}$, $P^{(1)}=0$, and $T^{(j,j)}=0$ for $j>1$,
     while when both $k$ and $\ell$ vanish then, for all $j$, $T^{(j,j)}=0$ and $P^{(j)}=I$.

   \item\label{item:2} The operators $\mathcal L$ and $\mathcal M$ in \eqref{eq:3pde} are allowed to have order
     zero, i.e., to be simply multiplication by fixed matrices, and then the operators in the lemma
     are simply the same operators. For example,
     \begin{equation}
       T^{(0,0)}=\mathcal L=
       \begin{pmatrix}
         0& 1& 0&0&0\\-1&0&0&0&0\\0&0&0&0&0\\0&0&0&0&0\\0&0&0&0&0
       \end{pmatrix},\qquad
  T^{(0,1)}=\mathcal M=
       \begin{pmatrix}
         0& 0& 0&a&b\\0&0&0&c&d\\0&0&i m&0&0\\-a&-c&0&0&0\\-b&-d&0&0&0
       \end{pmatrix}.
     \end{equation}
For these operators, 
\begin{gather*}
T^{(1,1)}=\begin{pmatrix}
         0& 0& 0&0&0\\0&0&0&0&0\\0&0&im&0&0\\0&0&0&0&0\\0&0&0&0&0
       \end{pmatrix}
\qquad\text{and}\qquad
T^{(2,2)} =\begin{pmatrix}
         0& 0& 0&0&0\\0&0&0&0&0\\0&0&0&0&0\\0&0&0&0&bc-ad\\0&0&0&ad-bc&0
       \end{pmatrix}.
\end{gather*}
When $ad-bc\ne0$ then all eigenvalues of $T^{(0,0)}+\mu T^{(0,1)}$ have been accounted for, so
$T^{(j,j)}=0$ for $j>2$. On the other hand, when $ad-bc=0$ then zero
is an eigenvalue of $T^{(0,0)}+\mu T^{(0,1)}$ with multiplicity two for all $\mu$ (with the
eigenvectors being $\lp\begin{smallmatrix} 0 &0 &0 &-b &a
\end{smallmatrix}\rp^T$ and $\lp\begin{smallmatrix}-c \mu &a\mu & 0& -(1+b) &a
\end{smallmatrix}\rp^T$ when $a\ne0$), so $T^{(j,j)}=0$ for $j>1$.
  \end{enumerate}
\end{example}

\subsection{Theorem and proof}\label{subs:mainTheorem}
The following projections and operator will appear in the statement and proof of the convergence
theorem. We assume that either \eqref{eq:delepsrel1} or \eqref{eq:delepsrel2} holds for some integer $s\ge
s_0$.

Let $\mathcal L$ and $\mathcal M$ be operators satisfying the conditions of
Assumption~\ref{sum:hyp}. Let $\widehat f(k)$ denote the Fourier transform of $f$ on $\RR^d$ or
$\TT^d$ and let $(g(k))\unhat$ denote the corresponding inverse Fourier transform of $g(k)$. 
Since $\mathcal L$ and $\mathcal M$ are constant-coefficient operators
there exist functions~$\widehat{\mathcal L}(k)$ and $\widehat{\mathcal M}(k)$ such that
$\widehat{\mathcal L f}=\widehat{\mathcal L}\widehat f$ and 
$\widehat{\mathcal M f}=\widehat{\mathcal M}\widehat f$. 

\begin{definition}\label{def:proj}
For any $k$, let $\widehat{\PP}(k)$ and $\widehat{\PP(\mu)}(k)$
denote the projections $\mathcal P(0)$ and $\mathcal P(\mu)$,
respectively, from Lemma~\ref{lem:pert}, where $p=s+1$ when \eqref{eq:delepsrel1} holds or $p=s+2$ when
\eqref{eq:delepsrel2} holds, $T^{(0,0)}\eqdef
\widehat{\mathcal L}(k)$, and $T^{(0,1)}\eqdef \widehat{\mathcal M}(k)$. Define the projection $\PP$
by $\PP f=(\widehat{\PP}(k)\widehat f(k))\unhat$, and the projection $\PP(\mu)$ by
$\PP(\mu) f=(\widehat{\PP(\mu)}(k)\widehat f(k))\unhat$. In addition, when \eqref{eq:delepsrel1}
holds then let $\widehat{T_{\lim}}(k)=C^s T^{(p,p)}$ ,   where $C$ is the constant
from~\eqref{eq:delepsrel1} and $T^{(p,p)}$ is from Lemma~\ref{lem:pert} with
$T^{(0,0)}$, $T^{(0,1)}$, and $p$ 
as mentioned above, and define the operator $T_{\lim}$ by 
$T_{\lim}f=(\widehat{T_{\lim}}(k)\widehat f(k))\unhat$. 
However, when \eqref{eq:delepsrel2} holds then define $T_{\lim}=0$.
\end{definition}

\begin{remark}\label{rem:tlim}
  Since $\widehat{\PP}(k)$ is an orthogonal projection for each $k$ and hence bounded by one, $\PP$
  is an orthogonal projection on $L^2$ and a bounded operator on $H^s$ for all $s$. In contrast,
  although $\widehat{T_{\lim}}(k)$ is a bounded operator for each $k$ the operator
  $T_{\lim}$ may be unbounded.  By Lemma~\ref{lem:pert},  $\widehat{T_{\lim}}(k)$ is skew-adjoint for each $k$
  so  $T_{\lim}$ is anti-symmetric. 
\end{remark}

\begin{theorem}
  Assume that the conditions of Theorem~\ref{thm:unifexist} hold, that $\del$ and $\eps$ 
tend to zero while obeying
either \eqref{eq:delepsrel1} or \eqref{eq:delepsrel2} for some integer $s\ge s_0$, 
and that $u_0(x,\del,\eps)$
  converges in $H^{s_0+1}$ to $u_{0,0}(x)$ in that limit.

Then the solution $u(t,x,\del,\eps)$ of the initial-value problem \eqref{eq:3pde},
$u(0,x,\del,\eps)=u_0(x,\del,\eps)$ converges to the unique solution~$U(t,x)$ 
belonging to  $L^\infty([0,T];H^{s_0+1})\cap \Lip([0,T];L^2)$ of 
\begin{align}
  \label{eq:limpde}
  \PP\big[ A_0(0)U_t&+\sum_{j=1}^d A_j(U)U_{x_j}+T_{\lim}U-F(t,x,U)\big]=0,
\\\label{eq:strain}
 \mathcal (I-\PP)U&=0
\\
\label{eq:liminit}
U(0,x)&=u_{0,0}(x),
\end{align}
where $\PP$ is the orthogonal projection operator from
Definition~\ref{def:proj} and $T_{\lim}$ is the operator defined there.
\end{theorem}

\begin{proof}
The uniform bound for the $\vertiv{\;}_{s_0+1,\eps,A_0}$ norm of the solution of \eqref{eq:3pde},
proven in Theorem~\ref{thm:unifexist} shows that 
$\max_{0\le t\le T}[\|u(t,\cdot)\|_{s_0+1}^2+\|u_t\|_0^2]^{1/2}\le 2M$, where $T$ and $M$ are as in
Theorem~\ref{thm:unifexist}. By Ascoli's theorem plus the weak-$*$ compactness of
$L^\infty([0,T];H^{s_0+1})$, for every sequence of values of $\del$ and $\eps$ tending to zero while
satisfying~\eqref{eq:eps0c1c2} there is a subsequence converging weak-$*$ in
$L^\infty([0,T];H^{s_0+1})$ and strongly in $C^0([0,T];L^2)$ to a limit $U(t,x)$ in
$L^\infty([0,T];H^{s_0+1})\cap \Lip([0,T];L^2)$. In particular, this convergence together with the
assumption on the convergence of the initial data show that \eqref{eq:liminit} holds. 

By interpolation between Sobolev spaces, the convergence and bounds obtained so far
imply that the subsequence also converges to $U$ in $C^0([0,T];H^{s_0+1-\mu})$ for any $\mu>0$,
and hence also in $C^0([0,T];C^1)$. This yields the convergence in at least $L^2$ of 
$A_0(\eps u)u_t+\sum_{j=1}^d A_j(u)u_{x_j}-F(t,x,u)$ to $A_0(0)U_t+\sum_j A_j(U)U_{x_j}-F(t,x,U)$.

Now apply the projection $\PP(\mu)$ from Definition~\ref{def:proj} with $\mu=\frac\del\eps$ 
to the PDE~\eqref{eq:3pde}, which yields
\begin{equation}
  \label{eq:p3pde}
  \tfr{\del}\PP(\tfrac{\del}{\eps})(\mathcal L+\tfrac{\del}{\eps}\mathcal M)u
=-\PP(\tfrac\del\eps)\lb A_0(\eps u)u_t+\sum_j A_j(u)u_{x_j}-F(t,x,u)\rb.
\end{equation}
As noted above, the expression in brackets on the right side of \eqref{eq:p3pde} converges in
$C^0([0,T];L^2)$ as $\del$ and $\eps$ tends to zero in the manner stated in the theorem.
Since $\mu\eqdef \frac\del\eps$ tends to zero in that limit,
the projection $\PP(\frac\del\eps)$ converges in the strong operator topology to $\PP$ in that limit
since the Fourier transform of the former is uniformly bounded and converges pointwise to the
Fourier transform of the latter, so for any $f\in L^2$, $\|[\PP(\mu)-\PP]f\|_{L^2}^2=\int
|[\widehat{\PP(\mu)}(k)-\widehat \PP(k)]\widehat f(k)|^2\,dk$ (or that expression with the integral
replaced by a sum if the spatial domain is periodic) tends to zero by \eqref{eq:pjmutopj}, 
the Dominated Convergence
Theorem and the fact that orthogonal projection operators do not increase vector length. 
Hence the entire right side of \eqref{eq:p3pde} converges in the above limit to $\PP\lb
A_0(0)U_t-\sum_j A_j(U)U_{x_j}-F(t,x,U)\rb$. This implies that the left side of \eqref{eq:p3pde}
also converges. 

When \eqref{eq:delepsrel1} holds then that relation plus the definition $p=s+1$ from
Definition~\ref{def:proj} imply that $\frac{\mu^p}{\del}=C^{p-1}(1+o(1))$. Hence
Lemma~\ref{lem:pert} shows that 
\begin{equation}\label{eq:pdelt}
[\PP(\mu)(\tfr{\del}(\mathcal
L+\mu\mathcal M)f]\dohat=C^{p-1}T^{(p,p)}\widehat f+o(1)=\widehat{T_{\lim}}\widehat f+o(1).
\end{equation}
Although the Fourier transform of $T_{\lim}$ may be unbounded as a function of the Fourier transform
variable, \eqref{eq:pdelt} together with the convergence of $u$ to $U$ shows that the Fourier
transform of the left side of \eqref{eq:p3pde} converges pointwise to the Fourier transform of
$T_{\lim}U$. The fact that that left side is known to converge in $L^2$ implies that its Fourier transform
also converges in $L^2$. Since the pointwise and $L^2$ limits of a sequence of functions must
coincide when both exist, the Fourier transform of the left side of \eqref{eq:p3pde} tends in
$L^2$ to the Fourier transform of $T_{\lim}U$, and hence that left side tends to $T_{\lim}U$.  The
reduction process also shows that the Fourier transform of $T_{\lim}$ is in the image of
$\widehat{\PP}(k)$ for each $k$, so rearranging the limit of \eqref{eq:p3pde} yields
\eqref{eq:limpde}. When \eqref{eq:delepsrel2} holds instead of \eqref{eq:delepsrel1} then that
relation plus the definition $p=s+2$ from Definition~\ref{def:proj} imply that
$\frac{\mu^p}{\del}=o(1)$, so \eqref{eq:pdelt} holds with $C$ replaced by zero, and again leads to
\eqref{eq:limpde} but with $T_{\lim}=0$.

Now define 
$\widehat{\mathcal T(\del,\eps)}(k)\eqdef \tfr{\del}(\widehat{\mathcal L}(k)+\tfrac\del\eps \widehat{\mathcal
  M}(k))$. From the Fourier transform of  \eqref{eq:3pde}, \eqref{eq:delepsrel1} or
\eqref{eq:delepsrel2}, and Lemma~\ref{lem:pert}, 
\begin{equation}
  \label{eq:4imphatu}
  \la\lp I-\widehat{\PP(\tfrac\del\eps)}\rp\widehat u(k)\ra
\le c z(\del,\eps)
\la\lp I-\widehat{\PP(\tfrac\del\eps)}\rp\widehat{\mathcal T(\del,\eps)}(k)
\lp I-\widehat{\PP(\tfrac\del\eps)}\rp\widehat u(k)\ra,
\end{equation}
where $z(\del,\eps)=\mu=\eps^{\fr s}$ by \eqref{eq:projlarge} when \eqref{eq:delepsrel1} holds, and 
$z(\del,\eps)=\frac\del{\mu^p} \cdot {\mu}=\lp\frac{\eps^{1+\fr s}}\del\rp^s$ by the definition of
$\widehat{\mathcal T(\del,\eps)}$ plus \eqref{eq:projlarge} and the definition of $p$ in terms of
$s$ when \eqref{eq:delepsrel2} holds.
In the former case $\eps^{\fr s}$ clearly tends to zero with $\eps$, and in the latter case 
$\lp\frac{\eps^{1+\fr s}}\del\rp^s$ tends to zero by \eqref{eq:delepsrel2}, i.e., in either case
$z(\del,\eps)\to0$ as $\del$ and $\eps$ tend to zero. Since
\begin{equation*}
\begin{aligned}
\bigg|\lp I-\widehat{\PP(\tfrac\del\eps)}\rp\widehat{\mathcal T(\mu)}(k)
\lp I-\widehat{\PP(\tfrac\del\eps)}\rp&\widehat u(k)\bigg|
=
\la\lp I-\widehat{\PP(\tfrac\del\eps)}\rp\widehat{\mathcal T(\mu)}(k)
\widehat u(k)\ra
\\&\le 
\la\widehat{\mathcal T(\mu)}(k)
\widehat u(k)\ra
\\&=
\Big|\Big( A_0(\eps u)u_t+\sum_{j=1}^d A_j(u)u_{x_j}-F(t,x,u)\Big)\dohat(k)\Big|
\\&=O(1),
\end{aligned}
\end{equation*}
taking the limit of \eqref{eq:4imphatu} as $\del$ and $\eps$ tend to zero while satisfying
\eqref{eq:delepsrel1} or \eqref{eq:delepsrel2}
shows that $\lp I-\widehat{\PP}(k)\rp \widehat U(k)=0$, which implies \eqref{eq:strain}.

To show that a solution of the given smoothness of \eqref{eq:limpde}\textendash\eqref{eq:liminit} is
unique, let $\QQ_R$ be the projection onto
the Fourier modes for which $\widehat{T_{\lim}}(k)$ is bounded by $R$.
Since $\widehat{T_{\lim}}(k)$
is finite for each $k$, the limit as $R\to\infty$ of $\QQ_R$ is the identity operator. Since $\QQ_R$ is a
projection onto Fourier modes it commutes with $T_{\lim}$, and $\PP$. Hence, for any $R$,
\begin{equation*}
\ip{\QQ_RU}{\PP T_{\lim}U}=\ip{\QQ_R\PP U}{T_{\lim}U}=\ip{\QQ_RU}{T_{\lim}U}=\ip{\QQ_RU}{T_{\lim}\QQ_RU}=0
\end{equation*}
by \eqref{eq:strain} plus the antisymmetry of~$T_{\lim}$. 
Taking the difference of \eqref{eq:limpde} for two solutions $U^{(1)}$ and $U^{(2)}$, defining
$U\eqdef U^{(1)}-U^{(2)}$, taking the $L^2$ inner product of the result with $\QQ_RU$ and letting $R$
tend to infinity therefore yields
\begin{equation}
\begin{aligned}\label{eq:Udiffest}
  0&=\lim_{R\to\infty}\Big<\QQ_RU,\PP \Big( A_0(0)U_t+\sum_j A_j(U^{(1)})U_{x_j}+\sum_j \lB
    A_j(U^{(2)}+U)-A_j(U^{(2)})\rB U^{(2)}_{x_j} 
\\&\qquad\qquad\qquad\qquad\qquad +T_{\lim}U-\lb F(t,x,U^{(2)}+U)-F(t,x,U^{(2)})\rb\Big)\Big>
\\&=\ip{U}{\PP(A_0(0)U_t+\sum_j A_j(U^{(1)})U_{x_j}+M(t,x,U^{(2)},\grad_x U^{(2)})U)}
\\&=\ip{\PP U}{A_0(0)U_t+\sum_j A_j(U^{(1)})U_{x_j}+M(t,x,U^{(2)},\grad_x U^{(2)})U}
\\&=\ip{U}{A_0(0)U_t+\sum_j A_j(U^{(1)})U_{x_j}+M(t,x,U^{(2)},\grad_x U^{(2)})U},
\end{aligned}
\end{equation}
where 
\begin{equation*}
\begin{aligned}
  &\sum_j \lB A_j(U^{(2)}+U)-A_j(U^{(2)})\rB U^{(2)}_{x_j}-\lb F(t,x,U^{(2)}+U)-F(t,x,U^{(2)})\rb
\\&=\int_0^1 \dd{s} 
\sum_j A_j(U^{(2)}+s U)U^{(2)}_{x_j}- F(t,x,U^{(2)}+sU)\,ds 
\\&=\lB \int_0^1 \sum_j \grad_v \lb A_j(v)U^{(2)}_{x_j}-F(t,x,v)\rB_{v=U^{(2)}+s U}\,ds \rB U
\\&\eqdef M(t,x,U^{(2)},\grad_x U^{(2)}) U.
\end{aligned}
\end{equation*}
Since the final expression in \eqref{eq:Udiffest} looks like the $L^2$ estimate for a symmetric
hyperbolic system, we obtain
\begin{equation}\label{eq:Ul2est}
\begin{aligned}
  0&=\tddt\lp\half\ip{U}{A_0(0)U}\rp-\ip{U}{\lB\sum_j \prt_{x_j}A_j(U^{(1)})+M(t,x,U^{(2)},\grad_x
    U^{(2)})\rB U}
\\&\ge \tddt \lp\half \ip{U}{A_0(0)U}\rp-K_1\ip{U}{U}
\\&\ge 
\tddt \lp\half \ip{U}{A_0(0)U}\rp-K_2\lp\half\ip{U}{A_0(0)U}\rp
\end{aligned}
\end{equation}
for some $K_1$ and $K_2$ depending on the $\|\,\|_{H^{s_0+1}}$ norms of $U^{(1)}$ and $U^{(2)}$ and
the constant $c_0$ from \eqref{eq:c0b0}. Estimate \eqref{eq:Ul2est} plus the initial
condition~\eqref{eq:liminit} imply that
\begin{equation*}
\half \ip{U}{A_0(0)U}\le (\half \ip{U(0)}{A_0(0)U(0)})e^{kt}=0,
\end{equation*}
which implies that $U\equiv0$, i.e., $U^{(1)}\equiv U^{(2)}$, yielding uniqueness.

As usual, the uniqueness of the limit implies that convergence holds as $\del$ and $\eps$
tend to zero while satisfying \eqref{eq:delepsrel1} or \eqref{eq:delepsrel2} 
without restricting to a subsequence. 
\qed\end{proof}

\begin{example}
  \begin{enumerate}
  \item Consider the PDE 
    \begin{equation}
      \label{eq:3pdeg1}
      \begin{pmatrix} u\\v \end{pmatrix}_t+\tfr{\eps^2}
      \begin{pmatrix}
        1&0\\0&0
      \end{pmatrix}
 \begin{pmatrix} u\\v \end{pmatrix}_x
+\tfr\eps
      \begin{pmatrix}
        0&1\\1&0
      \end{pmatrix}
 \begin{pmatrix} u\\v \end{pmatrix}_y=0.
    \end{equation}
The relationship $\del=\eps^2$ does not satisfy \eqref{eq:deleps} in dimension two. Nevertheless, as
noted in the introduction, the fact that the coefficient matrix of the time derivatives does not
depend on $u$ or $v$ implies that 
solutions of \eqref{eq:3pdeg1} satisfy uniform bounds.
Let $f(x,y)$ be a function whose gradient belongs to $H^3$, and take the initial data to be
$u(0,x,y)=u_0(x,y)\eqdef -\eps f_y$ and $v(0,x,y)=v_0(x,y)\eqdef f_x$. Then $u_t(0,x,y)=0$ and
$v_t(0,x,y)=f_{yy}$, i.e., the initial time derivative is bounded. Since the PDE is linear with
constant coefficients, it is convenient to express the limit equation in Fourier space. By
part~\ref{item:1} of Example~\ref{eg:tjj}, when $k\ne0$ then the limit is
$\widehat U(t,k,\ell)=0$, $\widehat V_t-\frac{i\ell^2}{k}\widehat V=0$, while for $k=0$ but $\ell\ne0$ the
limit is $\widehat U(t,0,\ell)=0=\widehat V(t,0,\ell)$ and for $k=0=\ell$ the limit is $\widehat
U_t(t,0,0)=0=\widehat V_t(t,0,0)$. The initial data for the limit are $\widehat U(0,k,\ell)=0$ and
$\widehat V(0,k,\ell)=i k \widehat f(k,\ell)$. When $k$ and $\ell$ are both nonzero the solution of
the limit equation is $\widehat U(t,k,\ell)=0$ and
\begin{equation}\label{eq:limV}
\widehat V(t,k,\ell)=ik e^{i\frac{\ell^2}k t}\widehat f(k,\ell),
\end{equation}
while when $k=0$ then the limit is $\widehat U=0=\widehat V$. When the
spatial domain is $\RR^2$ then $\widehat{T_{\lim}}(k)=\frac{-i\ell^2}k$ is unbounded but when the
domain is $\TT^2$ then it is bounded since $|k|\ge c$ on the set where it is nonzero. Even when the
spatial domain is $\RR^2$, the fact that
$\widehat V(t,k,\ell)$ contains a factor of $k$ ensures that $\widehat V_t$ is bounded, but 
$\widehat V_{tt}$ will be unbounded if $\widehat f(0,0)\ne0$. 
The limit solution \eqref{eq:limV}, which implies the limit equation satisfied by $V$, can be
verified by solving the equation for $\widehat U$ and $\widehat V$ exactly for $k\ne0$. This yields
\begin{equation*}
\widehat V=ik\widehat f(k,\ell)
\tfrac{e^{i\frac{(-k+\sqkel)}{2\eps^2}t}(k+\sqkel) -e^{i\frac{(-k-\sqkel)}{2\eps^2}t}(k-\sqkel)}{2\sqkel}
+O(\eps)
\end{equation*}
whose limit as $\eps\to0$ indeed yields \eqref{eq:limV}.

  \item Adding the term $-\alp \lp\begin{smallmatrix}1&0\\0&1\end{smallmatrix}\rp
\lp\begin{smallmatrix}u\\v\end{smallmatrix}\rp_y$ to \eqref{eq:3pdeg1} changes the limit equation
for $V$ to $\widehat V_t-\frac{i\ell^2}{k-\alp\ell}\widehat V=0$. If $\alp$ is irrational but
well-approximated by rationals then the term $\widehat{T_{\lim}}(k)=\frac{-i\ell^2}{k-\alp\ell}$
may not be bounded by $(|k|+|\ell|)^3$ as that expression tends to infinity, even in the
periodic case, so $T_{\lim}$ may not be a bounded operator from $H^3$ to $L^2$.


\item Consider the PDE $u_t+u_x+\tfr{\del}\mathcal L u+\tfr\eps \mathcal Mu=0$, where $\del=\eps^{3/2}$
  and $\mathcal L$ and $\mathcal M$ are the matrices discussed in Part~\ref{item:2} of
Example~\ref{eg:tjj}. Since the choice of the relationship between $\del$ and $\eps$ makes $s$ in
 \eqref{eq:delepsrel1} equal two and hence $p$ in Definition~\ref{def:proj} equal three, the
 projection $\PP$ is orthogonal to the non-zero eigenspace of $T^{(2,2)}$ as well as those of
 $T^{(0,0)}$ and $T^{(1,1)}$. The formula for
 $T^{(2,2)}$ in Part~\ref{item:2} of 
Example~\ref{eg:tjj} therefore shows that when $m\ne0$ and $ad-bc\ne0$ then the limit equation is simply $U=0$
while when $m\ne0$ but $ad-bc=0$  then the limit equation is that the first three components of
$U$ vanish and $\prt_t+\prt_x$ of its last two components equal zero. This shows that even the number of
nonzero components of the limit cannot be determined simply by looking at the number of components that
do not contain large terms nor even by first eliminating  all components having terms of order $\fr{\del}$ and
then eliminating those remaining components having terms of order $\tfr\eps$ not coming from
components already eliminated, which works for the system \eqref{eq:3pdeg1}.

\end{enumerate}

\end{example}

\section{Acknowledgements}
Ju is supported by the NSFC (Grants No.11571046, 11471028, 11671225). 
Cheng and Ju are supported by the UK Royal Society ``International Exchanges'' scheme (Award
No. IE150886).
{\clrr Ju and Schochet are supported by the ISF-NSFC joint research program (NSFC Grant No. 11761141008 and ISF Grant No. 2519/17).
The authors thank an anonymous referee for comments that lead to improvements in the exposition of the paper.}

\bibliographystyle{alpha}

\begin{thebibliography}{Maj84}

\bibitem[BK82]{mr665380}
G.~Browning and H.-O. Kreiss.
\newblock Problems with different time scales for nonlinear partial
  differential equations.
\newblock {\em SIAM J. Appl. Math.}, 42(4):704--718, 1982.

\bibitem[Che14]{mr3284568}
B.~Cheng.
\newblock Improved accuracy of incompressible approximation of compressible
  {E}uler equations.
\newblock {\em SIAM J. Math. Anal.}, 46(6):3838--3864, 2014.

\bibitem[Fri76]{friedman76:pde}
A.~Friedman.
\newblock {\em Partial Differential Equations}.
\newblock Krieger, Huntington, NY, 1976.

\bibitem[Gal98]{MR1661025}
I.~Gallagher.
\newblock Applications of {S}chochet's methods to parabolic equations.
\newblock {\em J. Math. Pures Appl. (9)}, 77(10):989--1054, 1998.

\bibitem[Gre97]{mr1459589}
E.~Grenier.
\newblock Pseudo-differential energy estimates of singular perturbations.
\newblock {\em Comm. Pure Appl. Math.}, 50(9):821--865, 1997.

\bibitem[Kat82]{MR678094}
T.~Kato.
\newblock {\em A short introduction to perturbation theory for linear
  operators}.
\newblock Springer-Verlag, New York, 1982.


\bibitem[KM81]{klainerman81:singlim}
S.~Klainerman and A.~Majda.
\newblock Singular limits of quasilinear hyperbolic systems with large
  parameters and the incompressible limit of compressible fluids.
\newblock {\em Commun. Pure Appl. Math}, 34:481--524, 1981.

\bibitem[KM82]{klainerman82:compress}
S.~Klainerman and A.~Majda.
\newblock Compressible and incompressible fluids.
\newblock {\em Commun. Pure Appl. Math.}, 35:629--653, 1982.

\bibitem[Maj84]{Majdabook84}
A.~Majda.
\newblock {\em Compressible fluid flow and systems of conservation laws in
  several space variables}, volume~53 of {\em Applied Mathematical Sciences}.
\newblock Springer-Verlag, New York, 1984.

\bibitem[MK03]{majda03:_system_multis_model_tropic}
A.~Majda and R.~Klein.
\newblock Systematic multiscale models for the tropics.
\newblock {\em J. Atmos. Sci.}, 60(393--408), 2003.

\bibitem[MS01]{metivier01:euler}
G.~M{\'e}tivier and S.~Schochet.
\newblock The incompressible limit of the non-isentropic euler equations.
\newblock {\em Arch. Ration. Mech. Anal}, 158:61--90, 2001.

\bibitem[Sch88]{schochet88:asympt}
S.~Schochet.
\newblock Asymptotics for symmetric hyperbolic systems with a large parameter.
\newblock {\em J. Diff. Eq.}, 75:1--27, 1988.

\bibitem[Sch94]{scho94f}
S.~Schochet.
\newblock Fast singular limits of hyperbolic {PDE}s.
\newblock {\em J. Differential Equations}, 114(2):476--512, 1994.

\end{thebibliography}

\end{document}